\documentclass[preprint,3p,times,authoryear]{elsarticle}

\usepackage{amssymb}
\usepackage{amsmath}
\usepackage{amsthm}
\usepackage{natbib}
\usepackage{diagbox}
\usepackage{float}
\usepackage{orcidlink}
\usepackage{comment}
\usepackage[shortlabels]{enumitem}
\usepackage{subfigure}
\usepackage{multirow}
\usepackage{booktabs} 
\usepackage{xcolor}

\theoremstyle{plain}
\newtheorem{theorem}{Theorem}[section]
\newtheorem{prop}[theorem]{Proposition}

\newtheorem{lemma}[theorem]{Lemma}

\theoremstyle{remark}
\newtheorem{remark}[theorem]{Remark}
\theoremstyle{definition}
\newtheorem{definition}[theorem]{Definition}
\newtheorem{example}[theorem]{Example}

\newcommand{\li}{\mathrm{Li}}

\newcommand{\espe}{\mathbb{E}}
\newcommand{\prob}{\mathbb{P}}

\newcommand{\dpoi}{\mathrm{Poi}}
\newcommand{\dge}{\mathrm{Geom}}
\newcommand{\dgeom}{\mathrm{Geom}^*}
\newcommand{\dbinneg}{\mathrm{NegBin}}

\newcommand{\toas}{\xrightarrow{\text{a.s.}}}

\newcommand{\Z}{\operatorname{\mathbb{Z}}}

\newcommand{\N}{\operatorname{\mathbb{N}}}

\newcommand{\dif}{\operatorname{d\!}}

\journal{arXiv}

\begin{document}

\begin{frontmatter}


\title{Estimating hazard rates from $\delta$-records in discrete distributions}

\author[1,2]{Martín Alcalde \orcidlink{0009-0007-4275-2998}}
\ead{malcalde@unizar.es}

\author[1,2]{Miguel Lafuente \orcidlink{0000-0001-8471-3224}}
\ead{miguellb@unizar.es}

\author[1,2]{F. Javier López \orcidlink{0000-0002-7615-2559}\corref{cor1}}
\ead{javier.lopez@unizar.es}

\author[3]{Lina Maldonado \orcidlink{0000-0003-1647-3462}}
\ead{lmguaje@unizar.es}

\author[1,2]{Gerardo Sanz \orcidlink{0000-0002-6474-2252}}
\ead{gerardo.sanz@unizar.es}

\affiliation[1]{organization={Department of Statistical Methods, University of Zaragoza},
    country={Spain}
    }
\affiliation[2]{organization={Institute for Biocomputation and Physics of Complex Systems (BIFI), University of Zaragoza},
    country={Spain}
    }
    
    \affiliation[3]{organization={Department of Applied Economics, University of Zaragoza},
    country={Spain}
    }

\cortext[cor1]{Corresponding author}
    
\begin{abstract}

This paper focuses on nonparametric statistical inference of the hazard rate function of discrete distributions based on $\delta$-record data. We derive the explicit expression of the maximum likelihood estimator and determine its exact distribution, as well as some important characteristics such as its bias and mean squared error. We then discuss the construction of confidence intervals and goodness-of-fit tests. The performance of our proposals is evaluated using simulation methods. Applications to real data are given, as well. The estimation of the hazard rate function based on usual records has been studied in the literature, although many procedures require several samples of records.  In contrast, our approach relies on a single sequence of $\delta$-records, simplifying the experimental design and increasing the applicability of the methods.
\end{abstract}

\begin{keyword}
Discrete hazard rate\sep records\sep $\delta$-records\sep near-records\sep nonparametric maximum likelihood estimation\sep nonparametric statistical inference\sep goodness-of-fit test.

\end{keyword}

\end{frontmatter}

\section{Introduction}\label{introdution}

Records are observations that are greater (or smaller) than all previous observations in a sequence. They arise naturally in various fields, including climate science, hydrology, and sports, among others. Since the pioneering work of \cite{Chandler52}, hundreds of research papers have explored their probabilistic properties and statistical applications. Excellent references on record theory and its applications include the monographs \cite{Arnold11,Ahs04,Ahs15}.

In the last decades, there has been growing interest in the applications of records in reliability. Records naturally appear in various reliability contexts where the extreme values of a sequence are of interest. A prominent example is destructive stress testing, where each unit is stressed only up to the minimum rupture force observed previously, yielding a sequence of lower records. Over the past few years, several studies have explored the use of records in reliability, particularly in estimating the stress-strength parameter---that is, the probability that a random unit does not fail when subjected to a random stress---for different families of distributions; see \cite{Saini25, Ahmadi24, Yu23, Tripathi22}. Other works have focused on estimating the parameters of these distributions within both frequentist and Bayesian frameworks; see, for instance, \cite{Khan24, Kumar24, Wang24, Abbas23, Zhao23, Pak19}. Recent research on statistical applications of records also includes fields such as quality control (\cite{Guo20}), actuarial science (\cite{Empacher24}) or climatology (\cite{Castillo24}).

Nonparametric statistics using record data has not been as extensively studied as its parametric counterpart. The first works in this area, \cite{Samaniego88, Gulati94,Gulati95}, developed estimation methods for the survival and probability density functions (see also \cite{Gulati03}). More recently, \cite{Arabi15} proposed alternative nonparametric estimators for the survival function. One reason for the limited use of records in nonparametric estimation may be the need to collect multiple samples to obtain reliable estimates. Indeed, the studies mentioned above consider situations where several independent samples of records have been collected. This is because, in a single sample, once a value is observed as an (upper) record, smaller values no longer appear, leaving many regions of the survival function unestimated. While collecting multiple samples may be feasible in some experiments, it is impractical in others. Moreover, due to the scarcity of records—recall that in an independent and identically distributed sequence of random variables, the expected number of observations before a new record occurs is infinite—it may take an extremely long time to gather sufficient data.

The scarcity of records is a well-known drawback of statistical inference based on records. For that reason, related statistics, such as $k$-records, generalized $k$-records or $c$-records, have been used for inference; see, e.g., \cite{Jafari25, Newer24, Vidovic24, Jose22, Schimmenti21, Wang20, Shafay17, Paul15}. Among these, $\delta$-records, first defined in \cite{Gouet07}, emerge as a natural generalization of records. Specifically, an observation is a $\delta$-record if it exceeds the previous maximum by at least $\delta$. When dealing with upper (lower) records, setting $\delta$ to a negative (positive) value increases the sample size. This suggests that $\delta$-records may enhance statistical inference. Indeed, parametric inference based on $\delta$-records has been shown to outperform standard record-based inference in both frequentist and Bayesian frameworks, for both discrete and continuous distributions; see \cite{Gouet12, Gouet14, Gouet20, LBSM13}, all of which focus on parametric models. The collection of $\delta$-records is also well-suited to destructive stress testing experiments, as described above. In such settings, units can be stressed slightly beyond the previous record (by $\vert\delta\vert$), producing a sequence of $\delta$-records. Thus, whenever a destructive stress testing experiment is used to generate records, a minor adjustment enables the collection of $\delta$-records. As demonstrated throughout the paper, the increased number of observations allows us to rely on a single sample rather than multiple independent samples, making the $\delta$-record approach particularly advantageous for industrial applications compared to the standard record-based method.
 
This paper addresses the nonparametric maximum likelihood estimation (NPMLE) of the hazard function using $\delta$-records. We focus on discrete distributions. Although reliability models often assume continuous distributions, the discrete case is also highly relevant, as failures are sometimes observed only at discrete inspection times or after a specific number of uses rather than being measured by clock time; see \cite{Nair18}. To the best of our knowledge, this is the first study on the use of $\delta$-records in nonparametric statistics. We have chosen to analyze the discrete case because it allows for explicit expressions for the estimators and their distributions. The treatment of the continuous case, which can be approached through truncation and smoothing, will be addressed in future work.

For the estimation of hazard rate function, we first express the likelihood of the sample of $\delta$-records in terms of this function. This likelihood, which takes the form of a product, is then used to derive an explicit expression for the maximum likelihood estimator (MLE). We prove that the MLE can attain any rational value in $[0,1)$ and determine its exact distribution. It is notable that the distribution of the estimator can be explicitly derived, especially considering that many theoretical results in nonparametric statistical inference rely on large-sample theory. From this distribution, we obtain explicit expressions for the bias and the mean squared error of the estimator. Additionally, we analyze the asymptotic behavior of the estimator as the sample size approaches infinity and establish its strong consistency. Finally, we show how our results can be applied to construct confidence intervals for the discrete hazard rate and to define goodness-of-fit tests. The performance of the methods is assessed through simulations, employing both Monte Carlo and parametric bootstrap techniques. We also present examples of applications to real data sets.

The paper is organized as follows. In Section \ref{secNot}, we introduce the main definitions and notation used throughout the paper. The expression for the MLE of the hazard rate function and its exact probability distribution are derived in Section \ref{section_likelihood}. This section also explores several properties of the estimator in both exact and asymptotic settings. Statistical inference procedures are proposed and assessed in Section \ref{section_practica}, followed by their application to real data in Section \ref{section_fixed_sample_size}. The paper concludes with a discussion of key findings and directions for future research (Section \ref{conc}). Additional auxiliary results used in the paper are provided in the Appendix.

\section{Notation and preliminaries}\label{secNot}

Throughout this paper, we work with sequences $\{X_n\}_{n\in\N}$ of independent and identically distributed (i.i.d.) random variables, where each $X_n$ takes values in $\mathbb{Z}_+ := \{0,1,2,\dots\}$, assuming infinite support. For every $i \in \mathbb{Z}_+$, let us denote by $P(i) $ the probability that $X_n = i$, and define the survival function at $i$ as $\overline{F}(i) := \prob(X_n>i)=\sum_{\ell>i} P(\ell)$. The hazard rate function at the point $j \in \mathbb{Z}_+$ is then given by
\begin{equation*}
h_j :=\prob(X_1=j\mid X_1\ge j)= \frac{P(j)}{\overline{F}(j-1)}.
\end{equation*}

The probability mass function $P(\cdot)$ and survival function $\overline{F}(\cdot)$ can be expressed in terms of the hazard rates as follows:
\begin{align*}
P(j) &= \begin{cases}
h_0 & \text{if } j=0,\\
h_j(1-h_{j-1}) \cdots (1-h_0) & \text{if } j > 0,
\end{cases}\\
\overline{F}(j) &= \begin{cases}
1 & \text{if } j=-1,\\
(1-h_j) \cdots (1-h_0) & \text{if } j \ge 0.
\end{cases}
\end{align*}

For any integers $\ell \ge j\ge0$, the conditional hazard rate, $h_{\ell\vert j}$, and conditional survival function, $\overline{F}(\ell\vert j)$, are defined as
\begin{align*}
h_{\ell\vert j} &:= \frac{P(\ell)}{\overline{F}(j-1)} = \begin{cases}
h_j & \text{if } \ell =j,\\
h_\ell(1-h_{\ell-1}) \cdots (1-h_j) & \text{if } \ell > j,
\end{cases}\\
&\overline{F}(\ell\vert j) := \frac{\overline{F}(\ell)}{\overline{F}(j-1)} = (1-h_\ell) \cdots (1-h_j).
\end{align*}

We borrow the notation for records, record times, etc., from \cite{Gouet14}. Let $\{M_n\}_{n\in\N}$ be the sequence of 
partial maxima defined by $M_n:=\max\{X_1, X_2, \dots, X_n\}$, $n\in\N$, and let $\{L_n\}_{n\in\N}$ be the sequence of record times, defined as $L_1:=1$, $L_n:=\min\left\{m>L_{n-1}:X_m>X_{L_{n-1}}\right\}$ for $n\ge2$. The sequence of records $\{R_n\}_{n\in\N}$ is defined as 
$R_n:=X_{L_n}$, $n\in\N$.
The concept of $\delta$-record was introduced in \cite{Gouet07}. Given $\delta\in\mathbb{R}$, the observation $X_n$ ($n\ge2$) is a $\delta$-record if
\begin{equation}
X_n > M_{n-1} + \delta.  \label{eq_def_delta_record}
\end{equation}
By convention, $X_1$ is always considered a $\delta$-record. It is straightforward to note that $\delta$-records with $\delta = 0$ correspond to standard records, while for $\delta < 0$, condition \eqref{eq_def_delta_record} is less restrictive than the standard record condition. As a result, when $\delta < 0$, $\delta$-records include more observations, effectively addressing the issue of record scarcity---a common limitation in record-based statistical inference from i.i.d. sequences. Throughout the rest of the paper, we consider only $\delta < 0$. Furthermore, since the random variables $X_i$ take values in the set of nonnegative integers, we can, without loss of generality, restrict ourselves to the case $\delta = -1, -2, \ldots$ Note that if $\delta = -1$, $\delta$-records coincide with weak records as defined in \cite{Vervaat73}, which have been extensively studied in the literature (see, for instance, \cite{Stepanov93, Gouet08, Castano13}). So as to improve readability, we write $k=-\delta-1$ along the paper; thus, $X_n$ is a $\delta$-record if and only if $X_n\ge M_{n-1}-k$.

Also, $\delta$-records are closely related to near-records as introduced in \cite{Balakrishnan2005}, which are observations that are close to being a record. Formally, an observation $X_n$ is a near-record with parameter $a >0$ if  $M_{n-1} - a < X_n \leq M_{n-1}$. Moreover, we say that
the near-record $X_n$ is associated with the $m$-th record $R_m$ if $L_m<n<L_{m+1}$; that is, it appears after the $m$-th record but before the $(m+1)$-th record in the sequence. Clearly, for $k=0,1,\ldots,$ an observation is a $\delta$-record (with $k=-\delta-1$) if and only if it is a record or a near-record with $a=k+1$.

We are interested in the nonparametric estimation of the discrete hazard rates $\mathbf{h}:=(h_0,h_1,\ldots)$ based on $\delta$-records. To this end, we use some results from \cite{Gouet14}, who find  and study the maximum estimator of the parameter of a Geometric distribution based on $\delta$-records.  As in that paper, we consider the sample 
\[
\mathbf{T}:= (\mathbf{R}, \mathbf{S},\mathbf{Y}),
\]
where $\mathbf{R}:= (R_1, \ldots, R_n)$ is the vector of the first $n$ record values; $\mathbf{S}:= (S_1,\ldots,S_n)$ is the vector of counts of near-records associated with each record, that is, $S_i$ is the number of near-records associated with $R_i$, and $\mathbf{Y} := (Y_1^1,\ldots,Y_1^{S_1};\ldots; Y_{n}^1,\ldots,Y_n^{S_n})$ are the values of the near-records, where $Y^j_i$ is the value of the $j$-th near-record associated with $R_i$.

The likelihood of $\mathbf{t} = (\mathbf{r}, \mathbf{s}, \mathbf{y})$, which is a realization of the sample $\mathbf{T}$ for a general distribution $F$ with values in $\mathbb{Z}_+$, is given in \cite{Gouet14}, Proposition 2.3, as
\begin{equation}
\mathcal{L}(\mathbf{t}) = \overline{F}(r_n)\prod_{i=1}^n \frac{P(r_i)}{\overline{F}(r_i+\delta)^{s_i+1}} \prod_{j=1}^{s_i} P(y^j_i), \label{eq_vero_2014}
\end{equation}
where $0 \leq r_1 < r_2 < \cdots < r_n$, $s_1, \ldots, s_n$ are nonnegative integers and $y^j_i \in \{r_i+\delta+1, r_i+\delta+2,\ldots,r_i\}$ for all $i=1,\ldots,n$, and $j=1,\ldots,s_i$ whenever $s_i \ge 1$.

Note that, as no assumption is made about the form of the hazard rates and no observations greater than $r_n$ are included in our sample, we cannot make any inference about the values $h_j$, for $j> r_n$. Thus, our goal is to find the MLE of $(h_0,h_1,\ldots,h_{r_n})$. Proposition \ref{EMV} below gives the explicit expression for the estimator.
We use the following notation. Given $k\in\Z_+$ and a sample $\mathbf{t}$, let, for $m \in \mathbb{Z}_+$ and $\ell = 0,\ldots,k$,

{\allowdisplaybreaks
\begin{align}
a_m^\ell &:= \begin{cases}
\displaystyle\sum_{i=1}^n\left(\mathbf{1}_{\{r_i=m\}} + \sum_{j=1}^{s_i}\mathbf{1}_{\{r_i=m,y_i^j=m\}}\right) &\text{if } \ell = 0,\\
\vspace{2mm}\displaystyle\sum_{i=1}^n\sum_{j=1}^{s_i}\mathbf{1}_{\{r_i=m+\ell,y_i^j=m\}} &\text{if } \ell = 1,\dots,k,
\end{cases}\label{a_from_t}\\
v_m^\ell &:= \sum_{j=0}^\ell a_m^j,\label{v_from_t}\\
n_m^\ell &:= 1 + \sum_{j=0}^{\ell} v_{m+j}^{\ell-j}.\label{n_from_t}
\end{align}
}
That is, $a_m^0$ is the number of times the value $m$ has been a weak record; while, for $\ell>0$, $a_m^\ell$ is the number of times the value $m$ has been equal to the current record minus $\ell$. Also,  $v_m^\ell$ is the number of times the value $m$ has been a near-record with parameter $\ell+1$, plus 1 if it has been a record. The term $n_m^\ell$ is $1$ plus the number of records whose value lies between $m$ and $m+\ell$, plus the total number of near-records with parameter $\ell+1$ associated with each of those records. We use the uppercase versions of $a_m^\ell$, $v_m^\ell$ and $n_m^\ell$ ($A_m^\ell$, $V_m^\ell$ and $N_m^\ell$) for the corresponding random variables.
\section{$\delta$-record-based NPMLE of the discrete hazard rate and its properties}\label{section_likelihood}
In this section, we derive the MLE of the hazard function and analyze its properties, including its exact distribution for fixed $k\in\Z_+$ and its asymptotic behavior as $k\to\infty$. The explicit expression for the MLE is presented in the next subsection.
\subsection{Expression for the estimator}
\begin{prop}\label{EMV}
The NPMLE of $\mathbf{h}$ based on the sample $\mathbf{T}$ is given by:
\begin{equation}
\hat{h}_{j,k} := \frac{V^k_j}{1+\sum\limits_{\ell=0}^k V^{k-\ell}_{j+\ell}} = \frac{V^k_j}{N^k_j}, \label{eq_estimador}
\end{equation}
for $j=0,1,\ldots,R_n$.
\end{prop}
\begin{proof}
We can express \eqref{eq_vero_2014} as a function of the discrete hazard rates $\mathbf{h}$ yielding
\begin{align}
\mathcal{L}(\mathbf{t}\mid \mathbf{h}) &= (1-h_{r_n})\cdots(1-h_0) \prod_{i=1}^{n} h_{r_i}(1-h_{r_{i}-1})\cdots(1-h_0)\prod_{j=1}^{s_i} h_{y^j_i}(1-h_{y^{j}_i - 1})\cdots(1-h_0)\notag\\
&=\prod_{j=0}^{r_{n}} h_j^{v^k_j}(1-h_j)^{1 + \sum\limits_{\ell=1}^k v_{j+\ell}^{k-\ell}}.\label{eq_verosimilitud}
\end{align}
The last equality follows by counting the number of times each value $0,\ldots,r_n$ may be found in $\mathbf{t}$.
Note that in the display above, the index $i$ in the first product corresponds to the index of records $1,\ldots,n$ and the index $j$ in the second product corresponds to the index of the near-records associated with $r_i$, which are $j=1,\ldots, s_i$, while the index $j$ in \eqref{eq_verosimilitud} refers to the integer values $0,\ldots,r_n$. Given that the likelihood function is expressed as the product of terms involving each of the hazard rates, it is immediate to check that \eqref{eq_verosimilitud} is maximized when each $h_j$, $j=0,\ldots,r_n$, takes the value 
\begin{equation*}\frac{v^k_j}{1+\sum\limits_{\ell=0}^k v^{k-\ell}_{j+\ell}},\end{equation*}
so the result is proved. 
\end{proof}

\begin{remark} The expression of the MLE of each $h_j$ has a simple interpretation in the case that the maximum of the sample ($r_n$) is greater than  $j+k$. This interpretation is given in terms of the variables $\tilde L_m=\min\{n\ge1:X_n>m\}$, $m\in\Z_+$, which represents the number of observations in the original sequence $X_1,\ldots,X_n,\ldots$ until an observation greater than $m$ is obtained. With this definition, the MLE of $h_j$ in \eqref{eq_estimador} can be expressed as
\begin{equation}
\hat{h}_{j,k}=\frac{\sum\limits_{i=1}^{\tilde L_{j+k}} \mathbf{1}_{\{X_i = j\}}}{ \sum\limits_{i=1}^{\tilde L_{j+k}} \mathbf{1}_{\{X_i \ge j\}}}. \label{estimator_form_2}
\end{equation}
That is, the numerator is the number of times the integer $j$ is observed as a $\delta$-record until a value greater than $j+k$ is observed, while the denominator is the number of $\delta$-records greater than or equal to $j$ until a value greater than $j+k$ is observed. In conclusion, if $r_n\ge j+k$, the MLE of $h_j$ based on $\delta$-records coincides with the classical nonparametric estimator (see \cite{Karlis07}) of the hazard rate obtained from the $\delta$-record sample stopped at the random time $\tilde L_{j+k}$. \end{remark}

\subsection{Distribution of $\hat{h}_{j,k}$}\label{Properties}
The main result of this paper (Theorem \ref{theorem_rhat}) establishes the exact distribution of the NPMLE $\hat{h}_{j,k}$. Its proof is lengthy and divided into several propositions that analyze the behavior of the random variables $V_m^\ell$ introduced earlier, which may also be of independent interest. These propositions, along with their proofs, follow Theorem \ref{theorem_rhat} within this section.

First, let us introduce the following notation. Given $k\in\N$ and $j\in\Z_+$, we define
\begin{equation*}
	d_{j,k} := 1-\sum_{i=1}^k h_{j+i\vert j}.
\end{equation*}
Observe that $d_{j,k}$ represents the probability that $X_i$ is either equal to $j$ or greater than $j+k$ conditioned on $\{X_i\ge j\}$. In other words, $d_{j,k}=h_j+\overline{F}(j+k\vert j)$.

\begin{theorem}\label{theorem_rhat}
Let  $k \in \mathbb{N}$ and $ j \in \mathbb{Z}_+$ such that $j+k\le r_n$. The estimator $\hat{h}_{j,k}$ takes values in $\mathbb{Q} \cap [0,1)$ and its probability mass function is
\begin{align*}
\mathbb{P}(\hat{h}_{j,k} = q) = (d_{j,k} - h_j) \sum_{i=1}^\infty \binom{bi-1}{ai}\ h_j^{ai} (1 - d_{j,k})^{(b-a)i-1},
\end{align*}
for all \( q = a/b \), with \(a, b \in \mathbb{Z}_+ \), $a<b$ and \( \gcd(a, b) = 1 \). Moreover, the cumulative distribution function \( G_{j,k} \) of \(\hat{h}_{j,k}\) on \([0,1)\) is given by
\begin{equation*}
G_{j,k}(x):= \prob(\hat{h}_{j,k} \le x) 
= (d_{j,k} - h_j) \sum_{m=0}^\infty (1 - d_{j,k})^m \sum_{\ell=0}^{\left\lfloor \frac{(1+m)x}{1-x} \right\rfloor}\ \binom{m+\ell}{\ell}\ h_j^\ell
\end{equation*}
for all \( x \in [0,1) \).
\end{theorem}

Before characterizing the distribution of $\hat{h}_{j,k}$, we need to establish the probabilistic properties of the random vector $(V^k_j,\ldots,V^0_{j+k})$. To this end, we introduce new random variables that are similar to those defined by Equations \eqref{a_from_t}, \eqref{v_from_t} and \eqref{n_from_t}. Assume that a sequence $\{X_n\}_{n\in\N}$ is given. Let us define
\begin{align}
\tilde{A}_m^\ell &:= \begin{cases}\mathbf{1}_{\{X_1 = m\}} + \displaystyle\sum_{i=2}^\infty \mathbf{1}_{\{X_i = m, \ M_{i-1} \le m\}} &\text{if } \ell = 0,\\
\displaystyle\sum_{i=2}^\infty\mathbf{1}_{\{X_i = m, \ m=M_{i-1}-\ell\}} &\text{if } \ell =1,\ldots,k,
\end{cases}\label{tilde_A}\\
\tilde V_m^\ell &:= \sum_{j=0}^\ell \tilde{A}_{m}^j,\label{tilde_V}\\
\tilde N_m^\ell &:= 1 + \sum_{j=0}^\ell \tilde{V}_{m+j}^{\ell-j},\notag
\end{align}
for all $m\in\mathbb{Z}_+$ and $\ell=0,1,\ldots,k$. The random variable $\tilde{A}_m^0$ counts the number of times that $m$ is a weak record, and $\tilde{A}_m^\ell$ for $\ell > 0$ counts the number of times that $m$ is equal to the previous maximum minus $\ell$ in the entire sequence. This is in contrast to the original variables $A_m^\ell$, which are computed from the sample $\mathbf{t}$. Let us clarify this apparently subtle difference. The MLE of $h_j$ in \eqref{eq_estimador} is a function of the random vector $(V_j^k,\ldots,V_{j+k}^0)$. When a sample $\mathbf{t}$ of the first $n$ records and their associated near-records is observed, the random variables $A_m^\ell$ (and therefore $V_m^\ell$ and $N_m^\ell$) coincide with the corresponding tilde version defined in \eqref{tilde_A} if $r_n\ge j+k$, since this conditions ensures that all the $\delta$-records with value $j$ in the entire sequence are included in our sample $\mathbf{t}$. On the other hand, when $r_n < j+k$, the value $j$ may be observed as a near-record associated with future records. In this case, it would not be included in our sample and some of the random variables $A_m^l$ may differ from $\tilde A_m^\ell$. Thus, $A_m^\ell$ may be seen as a censored value of $\tilde A_m^\ell$. In order to avoid extra difficulties in the derivation of $\hat{h}_{j,k}$, we assume that $r_n\ge j+k$, meaning that the sample contains at least one value greater than or equal to $j+k$. In this situation, all the values of $A_m^\ell, V_m^\ell$ and $N_m^\ell$ coincide with the values of $\tilde A_m^\ell, \tilde V_m^\ell$ and $\tilde N_m^\ell$. Therefore, we study the distribution of $\hat h_{j,k}$ by examining the distribution of variables $\tilde V_m^\ell$. In practice, in order to apply Theorem \ref{theorem_rhat} below, it is enough to sample until a record greater than $j+k$ is observed. See Remark \ref{jgrande}. 
To simplify the notation, in the rest of this section we write $A_m^\ell, V_m^\ell$ and $N_m^\ell$ instead of $\tilde A_m^\ell, \tilde V_m^\ell$ and $\tilde N_m^\ell$.

Firstly, we analyze the probabilistic properties of the random variables $A_m^\ell$. Note that there is a relationship between $V_m^\ell$ and $V_m^{\ell+1}$ defined as
\begin{equation}
V_m^{\ell+1} = A_m^{\ell+1} + V_m^\ell. \label{eq_relac_k+1_k}
\end{equation}
This recurrence will be used throughout this section. Now, note that the distribution of the vector $(V_j^k, \dots, V^0_{j+k})$ depends on the vector $\mathbf{A}_{j,k}$, given by
{\allowdisplaybreaks
\begin{align}\label{vectores}
(&A_j^0, A_{j+1}^0, \dots, A_{j+k-1}^0, A_{j+k}^0,\notag\\
&A_j^1, A_{j+1}^1, \dots, A_{j+k-1}^1,\notag\\
&\vdots \\
&A_j^{k-1}, A_{j+1}^{k-1},\notag\\
&A_j^k).\notag
\end{align}}

To characterize the distribution of $\mathbf{A}_{j,k}$, we need the joint distribution of the random vector in \eqref{vectores}. Unfortunately, the vectors in different lines of \eqref{vectores} are not independent. We can reorder the components of  $\mathbf{A}_{j,k}$ in a different way to make the analysis easier. In fact, the vector may be decomposed into the following mutually independent sub-vectors:
\begin{align*}
&A_j^0,\\
&(A_j^1, A_{j+1}^0),\\
&\vdots\\
&(A_j^{k-1}, A_{j+1}^{k-2}, \dots, A_{j+k-1}^0),\\
&(A_j^k, A_{j+1}^{k-1}, \dots, A_{j+k-1}^1, A_{j+k}^0).
\end{align*}

The behavior of the first variable $A_j^0$ is well-known, as it corresponds to the number of times the value $j$ is a weak record, which was analyzed in \cite{Stepanov93} (see Lemma \ref{lemma_Stepanov}). From this result, we have that
\begin{equation}
\prob(A_j^0 = a_j^0) = h_j^{a_j^0}(1-h_j), \quad a_j^0 \in \mathbb{Z}_+. \label{eq_prob_0}
\end{equation}

In the next result, we find the distribution of the remaining vectors, ultimately obtaining the joint distribution of the vector \eqref{vectores}. We adopt the classical notation for multinomial coefficients. For $i_1,\ldots,i_{n}\in\Z_+$,
$$
\binom{i_1+\cdots+i_n}{i_1,\ldots,i_n} = \frac{(i_1+\cdots+i_n)!}{i_1!\cdots i_n!}.
$$

\begin{prop}\label{prop_xi} For $\ell =1,\ldots,k$, the probability mass function of $(A_j^\ell, A_{j+1}^{\ell-1}, \dots, A_{j+\ell-1}^1, A_{j+\ell}^0)$
is 
\begin{align*}
\prob(&A_j^\ell = a^\ell_j, A_{j+1}^{\ell-1} = a_{j+1}^{\ell-1}, \ldots,A^{1}_{j+\ell-1} = a^{1}_{j+\ell-1},A^{0}_{j+\ell} = a^{0}_{j+\ell})\\
&=\begin{cases}
1 - h_{j+\ell} &\text{if } a_{j+m}^{\ell-m} = 0,m=0,\ldots,\ell,\\
h_{j+\ell}\ h_{j+\ell\vert j}^{a^0_{j+\ell}-1}\cdots h_{j+1\vert j}^{a_{j+1}^{\ell-1}}\ h_j^{a_j^\ell}\ \overline{F}(j+\ell\vert j)\ \displaystyle{a_j^{\ell} + a_{j+1}^{\ell-1}+\cdots + a_{j+\ell}^0 - 1\choose a_j^\ell,a_{j+1}^{\ell-1},\ldots, a_{j+\ell}^0 - 1}^{\phantom{\frac{1}{2}}} &\text{if }  a_{j+\ell}^{0} \ge 1, a_{j+m}^{\ell-m} \in\Z_+,m=1,\ldots,\ell.
\end{cases}
\end{align*}
\begin{proof}
First, observe that if $a_{j+\ell}^0 = 0$, then it follows that $a^{\ell-m}_{j+m} = 0$ for all $m$. Thus, from \eqref{eq_prob_0}, $\prob(A_{j+\ell}^0 = 0) = 1-h_{j+\ell}$. Let us consider the case  $a_{j+\ell}^0 \ge 1$. Let 
\begin{equation*}
E=\left\{A_j^\ell = a^\ell_j, A_{j+1}^{\ell-1} = a_{j+1}^{\ell-1}, \ldots,A^{1}_{j+\ell-1} = a^{1}_{j+\ell-1},A^{0}_{j+\ell} = a^{0}_{j+\ell}\right\}.
\end{equation*}
 For the event $E$ to occur, a record with value $j+\ell$ must be observed, which happens with probability $h_{j+\ell}$. Once this record is observed, $E$ occurs if and only if, before the arrival of a new record, we observe $a_{j+\ell}^0-1$ times the value $j+\ell$, $a_{j+\ell-1}$ times the value $j+\ell-1$, $\ldots$, and $a_j^\ell$ times the value $j$. Since these observations (and the new record) all have values greater than or equal to $j$, we can restrict ourselves to the random variables conditioned on being greater than or equal to $j$. This means that the probability of observing $j'\ge j$ is $h_{j'\vert j}$.  Therefore, the probability of $E$ is easily seen to be as stated in the proposition. 
\end{proof}
\end{prop}

Using the behavior of the vectors $(A_j^\ell,A_{j+1}^{\ell-1},\ldots,A_{j+\ell}^0)$ for all $\ell$, in Proposition \ref{theorem_dist_conjunta} below we derive the distribution of the vector $(V_j^k,\ldots,V_{j+k}^0)$, which is necessary for finding the distribution of $\hat h_{j,k}$. To proceed, we utilize the multidimensional Geometric distribution, which corresponds to the number of failures of different types before the first success in independent trials---see Definition \ref{def_geometric} in the Appendix.

Lemma \ref{propdistrgeom} in the Appendix states the properties of the multidimensional Geometric distribution that will be needed later.
We now find the distribution of $(V_j^1,V_{j+1}^0)$, which are used in the inductive proof of Proposition \ref{theorem_dist_conjunta} below. 

\begin{prop}
Let $j\in\Z_+$ such that $j+1\le r_n$. The vector $(V^1_{j},V^0_{j+1})$ follows a $\mathrm{Geom}^*(h_j, h_{j+1\vert j})$ distribution.\label{prop_k_1}
\begin{proof}
Let $j, a, b \in \mathbb{Z}_+$. We aim to verify that 
\begin{equation}
\prob(V^1_j = a,V^0_{j+1} = b) = h_j^a\ h_{j+1\vert j}^b\ \overline{F}(j+1\vert j)\ {a+b\choose a}.\label{geom_dum_2}
\end{equation}
First, if $b = 0$, then it follows  from \eqref{eq_relac_k+1_k} that $V^1_j = A^0_j$. Thus, by Lemma \ref{lemma_Stepanov}, 
\begin{equation*}
\prob(V^1_j = a, V^0_{j+1} = 0) = \prob(A^0_j = a, A^0_{j+1}=0) =h_j^a(1-h_j)(1-h_{j+1}) = h_j^a\ \overline{F}(j+1\vert j),
\end{equation*} which satisfies \eqref{geom_dum_2}. Now, for $b > 0$, by \eqref{eq_relac_k+1_k}, we have 
$$\{V^1_j = a, V^0_{j+1} = b\} = \bigcup_{a_0 + a_1 = a}\{V^0_j = a_0, A^1_j = a_1, V^0_{j+1} = b\}.
$$
To compute the probability of this event, we must consider the different combinations of $a_0$ and $a_1$ such that $a_0+a_1=a$. 
Note that the events $\{V_j^0=a_0\}$ and $\{A_j^1=a_1,V_{j+1}^0=b\}$ are independent. From Lemma \ref{lemma_Stepanov},
\begin{equation}
\prob(V^0_j = a_0) = h_j^{a_0}(1 - h_j). \label{bb}
\end{equation}
Also, from Proposition \ref{prop_xi},
\begin{equation}
\prob(A^1_j = a_1, V^0_{j+1} = b) = h_{j+1}\ h_j^{a_1}\ h_{j+1 \vert j}^{b-1}\ \overline{F}(j+1 \vert j)\ \binom{a_1 + b - 1}{b-1}. \label{bbb}
\end{equation}
Thus, gathering \eqref{bb} and \eqref{bbb}, we obtain 
\begin{align*}
\prob(V^1_j = a, V^0_{j+1} = b) &= \sum_{a_0 + a_1 = a} \prob(V^0_j = a_0)\ \prob(A^1_j = a_1, V^0_{j+1} = b) \\
&= \sum_{a_1 = 0}^a h_j^a\ h_{j+1 \vert j}^b\ \overline{F}(j+1 \vert j)\ \binom{a_1 + b - 1}{b - 1} \\
&= h_j^a\ h_{j+1 \vert j}^b\ \overline{F}(j+1 \vert j)\ \binom{a+b}{a},
\end{align*}
where the last step comes from the identity 
\begin{equation*}
\binom{a+b}{a} = \sum_{a_1=0}^a \binom{a_1 + b - 1}{b-1}, \quad a \geq 0, \quad b \geq 1, \label{igualdad_a_b}
\end{equation*}
which is a particular case of the identity in Lemma \ref{lemma_combinatoria} in the Appendix.
\end{proof}
\end{prop}

The proof of Proposition \ref{theorem_dist_conjunta} below relies on constructing an inductive argument to derive the distribution of the vector $(V^k_j, \dots, V^0_{j+k})$ for any $k$. The core idea is to split the probability into two independent components: one associated with the values observed before $j+k$ becomes a record, and the other corresponding to the values after that. Specifically, these two parts involve the random variables $V_j^{k-1}, \dots, V_{j+k-1}^0$ and $A_j^k, \dots, A_{j+k-1}^1, V^0_{j+k}$, respectively. 

\begin{prop}\label{theorem_dist_conjunta}
Let $j,k\in\Z_+$ such that $j+k\le r_n$. Then, $(V^k_j,\dots,V^0_{j+k})\sim\mathrm{Geom}^*(h_j, h_{j+1 \vert j}, \dots, h_{j+k \vert j})$.
\begin{proof}
We proceed by induction on $k$. For $k=0,1$, the result follows from Lemma \ref{lemma_Stepanov} and Proposition \ref{prop_k_1}. Now, assume that the result holds for $0,1,\ldots,k-1$ and let us prove that it also holds for $k$. Let $k\ge2$ and $a_{j+i} \in \mathbb{Z}_+$ for all $i = 0, \dots, k$. We aim to prove that
\begin{equation}
\prob(V^k_j = a_j, V^{k-1}_{j+1} = a_{j+1}, \dots, V^0_{j+k} = a_{j+k}) 
= h_j^{a_j}\ h_{j+1 \vert j}^{a_{j+1}} \cdots h_{j+k \vert j}^{a_{j+k}}\ \overline{F}(j+k \vert j)\ \binom{a_j + a_{j+1} + \cdots + a_{j+k}}{a_j, a_{j+1}, \dots, a_{j+k}}. \label{eq_prob_k_gen}
\end{equation}

If $a_{j+k} = 0$, the formula follows immediately from the induction hypothesis, since $V^{k-i}_{j+i} = V^{k-i-1}_{j+i}$ for all $i = 0, \dots, k-1$, thereby yielding
\begin{align*}
\prob(&V^k_j = a_j, V^{k-1}_{j+1} = a_{j+1}, \dots, V^1_{j+k-1} = a_{j+k-1}, V^0_{j+k} = 0) \\
&= \prob(V_j^{k-1} = a_j, V_{j+1}^{k-2} = a_{j+1}, \dots, V^0_{j+k-1} = a_{j+k-1}, V_{j+k}^0 = 0) \\
&= h_j^{a_j}\ h_{j+1 \vert j}^{a_{j+1}} \cdots h_{j+k-1 \vert j}^{a_{j+k-1}}\ \overline{F}(j+k-1 \vert j)\ \binom{a_j + a_{j+1} + \cdots + a_{j+k-1}}{a_j, a_{j+1}, \dots, a_{j+k-1}}\ (1 - h_{j+k}),
\end{align*}
which matches \eqref{eq_prob_k_gen}. 

Now, consider the case $a_{j+k} \geq 1$. Applying \eqref{eq_relac_k+1_k}, we proceed as follows: for each $i = 0, \dots, k-1$, fix $V^{k-i-1}_{j+i} = a'_{j+i}$ and $A_{j+i}^{k-i} = b_{j+i}$. Recall that  $V^{k-i-1}_{j+i}$ represents the number of times that the value $j+i$ occurs before $j+k$ becomes a record, while $A_{j+i}^{k-i}$ counts the occurrences of $j+i$ while $j+k$ is the current record. By the induction hypothesis, the probability of the event $\{V_{j}^{k-1} = a'_j,\ldots,V_{j+k-1}^{0} = a'_{j+k-1}\}$ is given by
\begin{equation}
h_j^{a'_j} \cdots h_{j+k-1 \vert j}^{a'_{j+k-1}}\ \overline{F}(j+k-1 \vert j)\ \binom{a'_j + \cdots + a'_{j+k-1}}{a'_j, \dots, a'_{j+k-1}}. \label{j}
\end{equation}
Moreover, the probability of observing the record $j+k$, obtaining the remaining $b_{j+i}$ values for $j+i$ (for $i=0,\dots,k-1$), and observing $a_{j+k}-1$ additional occurrences of $j+k$ as a weak record is, by Proposition \ref{prop_xi},
\begin{equation}
h_{j+k}\ h_j^{b_j} \cdots h_{j+k-1 \vert j}^{b_{j+k-1}}\ h_{j+k \vert j}^{a_{j+k}-1}\ \overline{F}(j+k \vert j)\ \binom{b_j + \cdots + b_{j+k-1} + a_{j+k} - 1}{b_j, \dots, b_{j+k-1}, a_{j+k}-1}. \label{jj}
\end{equation}
Combining \eqref{j} and \eqref{jj} over all possible combinations of $a'_{j+i}$, we obtain
\begin{align}
\prob(&V^k_j = a_j, V^{k-1}_{j+1} = a_{j+1}, \dots, V^0_{j+k} = a_{j+k}) \nonumber \\
&= h_j^{a_j} \cdots h_{j+k-1 \vert j}^{a_{j+k-1}}\ h_{j+k \vert j}^{a_{j+k}}\ \overline{F}(j+k \vert j)  \nonumber \\
&\quad \times \sum_{a'_{j+k-1} + b_{j+k-1} = a_{j+k-1}} \cdots \sum_{a'_j + b_j = a_j} \binom{a'_j + \cdots + a'_{j+k-1}}{a'_j, \dots, a'_{j+k-1}}\ \binom{b_j + \cdots + b_{j+k-1} + a_{j+k} - 1}{b_j, \dots, b_{j+k-1}, a_{j+k}-1}. \label{jjj}
\end{align}
Finally, by applying Lemma \ref{lemma_combinatoria}, we observe that \eqref{jjj} simplifies to \eqref{eq_prob_k_gen}, thus completing the proof.
\end{proof}
\end{prop}

From the distribution of the vector $(V^k_j, \dots, V^0_{j+k})$ we derive the following result which is crucial in the proof of Theorem \ref{theorem_rhat}. 

\begin{prop}\label{corollary_n_jk} Let $k\in\N$ and $j\in\Z_+$ such that $j+k\le r_n$. We have that
\begin{enumerate}[a)]
	\item \label{corollary_n_jk_a}$$
	\left(V_j^k,\sum_{\ell=1}^k V_{j+\ell}^{k-\ell}\right)\sim\dgeom(h_j,1-d_{j,k}).
	$$
	\item The random variable $N_j^k$ follows a Geometric distribution (starting at $1$) with parameter $d_{j,k}-h_j$, that is, \label{corollary_n_jk_b}
	$$
	N_j^k\sim\dge(d_{j,k}-h_j).
	$$
\end{enumerate}
\begin{proof}
Both statements follow directly from Proposition \ref{theorem_dist_conjunta} and the properties of the multidimensional Geome\-tric distribution (see Lemma \ref{propdistrgeom}).
\end{proof}
\end{prop}

Proposition \ref{corollary_n_jk} is the key result for proving Theorem \ref{theorem_rhat}.
\begin{proof} [Proof of Theorem \ref{theorem_rhat}] Given \( q = a/b \), where \( a, b \in \mathbb{Z}_+ \) and \( a < b \), with \( \gcd(a, b) = 1 \), we can express \( q \) as
\[
q = \frac{\ell_i}{1 + \ell_i + m_i},
\]
where
\begin{equation}
(\ell_i, m_i) = (ai, -1 + (b - a)i), \quad i \in \mathbb{N}. \label{eq_solutions}
\end{equation}
These are all the possible solutions for \( \ell_i, m_i \in \mathbb{Z}_+ \). To verify this, observe that finding integers \( \ell \) and \( m \) such that \( q = \frac{\ell}{1 + \ell + m} \) is equivalent to solving the linear Diophantine equation
\[
am - (b - a)\ell = -a,
\]
for \( \ell \) and \( m \), whose general solution is given by \eqref{eq_solutions} (see Lemma \ref{lemma_diofanto} in the Appendix). Now, using Proposition \ref{corollary_n_jk} \ref{corollary_n_jk_a}, we may compute the probability distribution of the estimator \(\hat{h}_{j,k}\). Specifically,
{\allowdisplaybreaks
\begin{align*}
\mathbb{P}(\hat{h}_{j,k} = q) &= \sum_{i=1}^\infty \mathbb{P}\left( \hat{h}_{j,k} = \frac{\ell_i}{1 + \ell_i + m_i}, V^{k-1}_{j+1} + \cdots + V^0_{j+k} = m_i \right)\\
&= \sum_{i=1}^\infty \mathbb{P}\left( V^k_j = \ell_i, V^{k-1}_{j+1} + \cdots + V^0_{j+k} = m_i \right)\\
&= (d_{j,k} - h_j) \sum_{i=1}^\infty \binom{bi-1}{ai}\ h_j^{ai} (1 - d_{j,k})^{(b-a)i-1}.
\end{align*}}
To derive the cumulative distribution function \( G_{j,k} \), we proceed similarly. For any \( x \in [0, 1) \), we have
{\allowdisplaybreaks
\begin{align*}
G_{j,k}(x) &= \sum_{m=0}^\infty \mathbb{P}(\hat{h}_{j,k} \leq x, V^{k-1}_{j+1} + \cdots + V^0_{j+k} = m)\\
&= \sum_{m=0}^\infty \mathbb{P}\left( V^k_j \leq \left\lfloor \frac{(1+m)x}{1-x} \right\rfloor, V^{k-1}_{j+1} + \cdots + V^0_{j+k} = m \right)\\
&= (d_{j,k} - h_j) \sum_{m=0}^\infty \sum_{\ell=0}^{\left\lfloor \frac{(1+m)x}{1-x} \right\rfloor} \binom{m+\ell}{\ell}\ h_j^\ell (1 - d_{j,k})^m.
\end{align*}}
This concludes the proof of Theorem \ref{theorem_rhat}.
\end{proof}
\subsection{Properties of the estimator $\hat{h}_{j,k}$ for fixed $k$}

Theorem \ref{theorem_rhat} allows us to study certain properties of the estimator, such as its expectation, variance, and mean squared error (MSE), exactly rather than asymptotically. These computations provide insight into the practical performance of the estimator, as they enable a deeper understanding of its dependence on the parameter $k$.

In the proof of the next proposition, we  use two identities involving infinite series and integrals, which are stated and proved in Lemma \ref{lemma_suma_series_potencias} in the Appendix. Let $\li_2(x)=\sum_{n\ge 1} n^{-2}x^n$, for $x\in[-1,1]$, be the dilogarithm function.

\begin{prop}\label{prop_esperanza} 
Let $k \in \mathbb{N}$ and $j \in \mathbb{Z}_+$ such that $j+k\le r_n$. Then:
\begin{enumerate}[a)]
\item \label{prop_esperanza_apartado_a} 
\[
\mathbb{E} (\hat{h}_{j,k}) = \frac{(d_{j,k} - h_j)h_j}{(1 + h_j - d_{j,k})^2} \log{(d_{j,k} - h_j)} + \frac{h_j}{1 + h_j - d_{j,k}}.
\]
\item 
\begin{align*}
\mathrm{Var}(\hat{h}_{j,k}) &= \frac{(d_{j,k}-h_j)(h_j+d_{j,k}-1)h_j}{(1 + h_j - d_{j,k})^3} \mathrm{Li}_2(1+h_j-d_{j,k}) - \frac{(d_{j,k}-h_j)^2 h_j^2}{(1 + h_j - d_{j,k})^4} (\log{(d_{j,k}-h_j)})^2 \\
&\qquad - \frac{(d_{j,k}-h_j)(1-d_{j,k})h_j}{(1+h_j-d_{j,k})^3} \log{(d_{j,k}-h_j)}.
\end{align*}
\end{enumerate}

Consequently, the mean squared error of the estimator $\hat{h}_{j,k}$ is
\begin{align}\label{formula_MSE}
\mathrm{MSE}(\hat{h}_{j,k}) &= \frac{(d_{j,k}-h_j)(h_j+d_{j,k}-1)h_j}{(1 + h_j - d_{j,k})^3} \mathrm{Li}_2(1+h_j-d_{j,k}) + \frac{h_j^2(d_{j,k}-h_j)^2}{(1+h_j-d_{j,k})^2} \nonumber \\
&\qquad + \frac{(d_{j,k}-h_j)(2h_j(d_{j,k}-h_j)-(1-d_{j,k}))h_j}{(1 + h_j - d_{j,k})^2} \log{(d_{j,k}-h_j)}.
\end{align}
\end{prop}

\begin{proof}
\begin{enumerate}[a)]
\item The expectation is calculated directly from the cumulative distribution function:
{\allowdisplaybreaks
\begin{align*}
\espe(\hat{h}_{j,k}) &= \int_0^1 (1-G_{j,k}(x))\dif x\nonumber\\
&= 1 - (d_{j,k}-h_j)\sum_{m= 0}^\infty (1-d_{j,k})^m \sum_{i=0}^\infty \int_{\frac{i}{i+m+1}}^{\frac{i+1}{i+m+2}}\sum_{l=0}^{i} {m+\ell\choose \ell}\ h_j^\ell\dif x\nonumber\\
&= 1 - (d_{j,k}-h_j)\sum_{m= 0}^\infty (m+1)(1-d_{j,k})^m \sum_{\ell=0}^{\infty} \left[{m+\ell\choose \ell}\ h_j^\ell\sum_{i=\ell}^\infty \left(\frac{1}{i+m+1}-\frac{1}{i+m+2}\right)\right]\nonumber\\
&= 1 - (d_{j,k}-h_j)\sum_{m= 0}^\infty (m+1)(1-d_{j,k})^m \sum_{\ell=0}^{\infty} {m+\ell\choose \ell}\ h_j^\ell\frac{1}{\ell+m+1}.\nonumber\\
\end{align*}}Now, applying Lemma \ref{lemma_suma_series_potencias} \ref{lemma_ssp_aparatado_a}, we obtain the desired result.
{\allowdisplaybreaks
\begin{align}
\espe(\hat{h}_{j,k})
&= 1 - (d_{j,k}-h_j)\sum_{m= 0}^\infty (m+1)\frac{(1-d_{j,k})^m}{h_j^{m+1}}\int_0^{h_j}\frac{u^m}{(1-u)^{m+1}}\dif u\nonumber\\
&= 1-\frac{d_{j,k}-h_j}{h_j}\int_0^{h_j} \frac{1}{1-u}\sum_{m=0}^\infty (m+1)\left(\frac{(1-d_{j,k})u}{h_j(1-u)}\right)^m\dif u\nonumber\\
&= 1-\frac{d_{j,k}-h_j}{h_j}\int_0^{h_j}\frac{1-u}{(1-\frac{1+h_j-d_{j,k}}{h_j}u)^2}\dif u\nonumber\\
&= \frac{(d_{j,k} - h_j)h_j}{(1 +h_j - d_{j,k})^2}\log{(d_{j,k} - h_j)} + \frac{h_j}{1+h_j -d_{j,k}}.\nonumber
\end{align}
}
\item As in part \ref{prop_esperanza_apartado_a},  the cumulative distribution function allows us to compute the second moment.{\allowdisplaybreaks
\begin{align*}
\espe (\hat{h}^2_{j,k})  &= \int_0^1 2x (1-G_{j,k}(x)) \dif x \nonumber\\
&= 1 - (d_{j,k}-h_j) \sum_{m=0}^\infty (1-d_{j,k})^m \sum_{i=0}^\infty \sum_{\ell=0}^i \left[{m+\ell \choose \ell}\ h_j^\ell  \left( \frac{(i+1)^2}{(i+m+2)^2} - \frac{i^2}{(i+m+1)^2} \right) \right] \nonumber\\
&= 1 - (d_{j,k}-h_j) \sum_{m=0}^\infty (1-d_{j,k})^m \sum_{\ell=0}^\infty \left[{m+\ell \choose \ell}\ h_j^\ell   \sum_{i=\ell}^\infty \left( \frac{(i+1)^2}{(i+m+2)^2} - \frac{i^2}{(i+m+1)^2} \right) \right] \nonumber\\
&= 1 - (d_{j,k}-h_j) \sum_{m=0}^\infty (1-d_{j,k})^m \sum_{\ell=0}^\infty {m+\ell \choose \ell}\ h_j^\ell \left( 1 - \frac{\ell^2}{(\ell+m+1)^2} \right) \nonumber\\
&= 1 - \left(G_{j,k}(1) - (d_{j,k}-h_j) \sum_{m=0}^\infty (1-d_{j,k})^m \sum_{\ell=0}^\infty {m+\ell \choose \ell}\ h_j^\ell \frac{\ell^2}{(\ell+m+1)^2} \right) \nonumber\\
&= (d_{j,k}-h_j) \sum_{m=0}^\infty (1-d_{j,k})^m \sum_{\ell=0}^\infty {m+\ell \choose \ell}\ h_j^\ell \frac{\ell^2}{(\ell+m+1)^2}.
\end{align*}}At this point, we apply Lemma \ref{lemma_suma_series_potencias} \ref{lemma_ssp_aparatado_b}, yielding
{\allowdisplaybreaks
\begin{align*}
\espe (\hat{h}^2_{j,k}) &= (d_{j,k}-h_j) \sum_{m=0}^\infty (1-d_{j,k})^m \frac{(m+1)}{h_j^{m+1}} \int_0^{h_j} \frac{u^{m+1}(\log{(h_j)} - \log{(u)})}{(1-u)^{m+3}} \dif u \nonumber\\
&\quad + (d_{j,k}-h_j) \sum_{m=0}^\infty (1-d_{j,k})^m \frac{(m+1)^2}{h_j^{m+1}} \int_0^{h_j} \frac{u^{m+2}(\log{(h_j)} - \log{(u)})}{(1-u)^{m+3}} \dif u \nonumber\\
&= \frac{d_{j,k}-h_j}{h_j} \int_0^{h_j} \frac{u(\log{(h_j)} - \log{(u)})}{(1-u)(1-\frac{1+h_j-d_{j,k}}{h_j}u)^2} \dif u  +\frac{d_{j,k}-h_j}{h^2_j} \int_0^{h_j} \frac{u^2(\log{(h_j)} - \log{(u)})}{(1-u)(1-\frac{1+h_j-d_{j,k}}{h_j}u)^3} \dif u \nonumber\\
&= \frac{(d_{j,k}-h_j)(h_j+d_{j,k}-1)h_j}{(1+h_j-d_{j,k})^3} \li_2(1+h_j-d_{j,k}) + \frac{(d_{j,k}-h_j)(2h_j+d_{j,k}-1)h_j}{(1+h_j-d_{j,k})^3} \log{(d_{j,k}-h_j)} \\
&\quad+ \frac{h_j^2}{(1+h_j-d_{j,k})^2}. \nonumber
\end{align*}
}
From here, the result follows directly from part \ref{prop_esperanza_apartado_a}.
\end{enumerate}
\end{proof}
\subsection{Asymptotic properties of $\hat{h}_{j,k}$ as $k\to\infty$}
Regarding the asymptotic properties of the estimator, let us return to its MSE (see Equation  \eqref{formula_MSE}). Since $d_{j,k}=h_j+\overline F(j+k\vert j)$, it is clear that $d_{j,k}\to h_j$ as $k\to\infty$. Thus, the right-hand side of \eqref{formula_MSE} converges to 0 as $k\to\infty$, implying that the MSE of $\hat h_{j,k}$ is small for large $k$. This result is consistent with the fact that larger values of $k$ include more data in the sample and, consequently, should lead to more accurate estimates. However, note that in order for $k\to\infty$, we must also let $n\to\infty$ because, for fixed $n$, the sample of $\delta$-records remains fixed once $k>r_n$. Therefore, we choose $n=n_k>j+k$. As our observations are nonnegative integers, this choice guarantees that $r_n\ge j+k$, ensuring that the assumptions of Proposition \ref{prop_esperanza} hold. The next result establishes the strong consistency of our estimator in this setting.

\begin{theorem}\label{as_convergence} Let $n=n_k>j+k$. Then, $\hat{h}_{j,k}\toas h_j$ and $\mathrm{MSE}(\hat h_{j,k})\to0$ as $k\to\infty$.
\begin{proof}
It follows from Equation \eqref{estimator_form_2} and the fact that the random variable $\tilde{L}_{j+k}\toas \infty$ as $k\to\infty$, since
$$
\hat{h}_{j,k} = \frac{\frac{1}{\tilde{L}_{j+k}}\sum\limits_{i=1}^{\tilde{L}_{j+k}} \mathbf{1}_{\{X_i=j\}}}{\frac{1}{\tilde{L}_{j+k}}\sum\limits_{i=1}^{\tilde{L}_{j+k}} \mathbf{1}_{\{X_i\ge j\}}}\toas \frac{\prob(X_1=j)}{\prob(X_1\ge j)}\;\text{ as }\; k\to \infty.
$$
Finally, since $0\le\hat{h}_{j,k}< 1$, almost sure convergence implies $L^2$ convergence, or equivalently, that $\mathrm{MSE}(\hat{h}_{j,k})\to 0$ as $k\to\infty$.
\end{proof}
\end{theorem}

\begin{example}[Geometric distribution]
In the case where the distribution $F$ is Geometric with parameter $p \in (0,1)$, we have that $h_j = p$ and $d_{j,k} = p + (1 - p)^{k+1}$ for all $j \in \mathbb{Z}_+$, $k \in \mathbb{N}$. Hence, from Proposition \ref{prop_esperanza} \ref{prop_esperanza_apartado_a}, we have that
\begin{equation}\label{espe_geo}
	\espe(\hat h_{j,k})=\frac{(k+1)(1-p)^{k+1}p}{(1-(1-p)^{k+1})^2}\log{(1-p)}+\frac{p}{1-(1-p)^{k+1}}.\end{equation}
The variance and MSE can be computed analogously.
For this distribution, the expected value of $\hat h_{j,k}$ does not depend on $j$. This is consistent with the fact that, in the Geometric distribution, the hazard rate $h_j$ is constant ($h_j$ equals the probability of success $p$). Note, however, that the MLE is biased. Figure \ref{figure_plot_esperanza} plots \eqref{espe_geo} as a function of $p$ (blue lines) and the identity function (black line) for several values of $k$, showing that, indeed, the MLE has a negative bias. The graph also shows that, as $k$ increases, \eqref{espe_geo} approaches $p$, as expected due to Theorem \ref{as_convergence}. 

\begin{figure}[h]
	\centering
	\includegraphics[scale = 0.6]{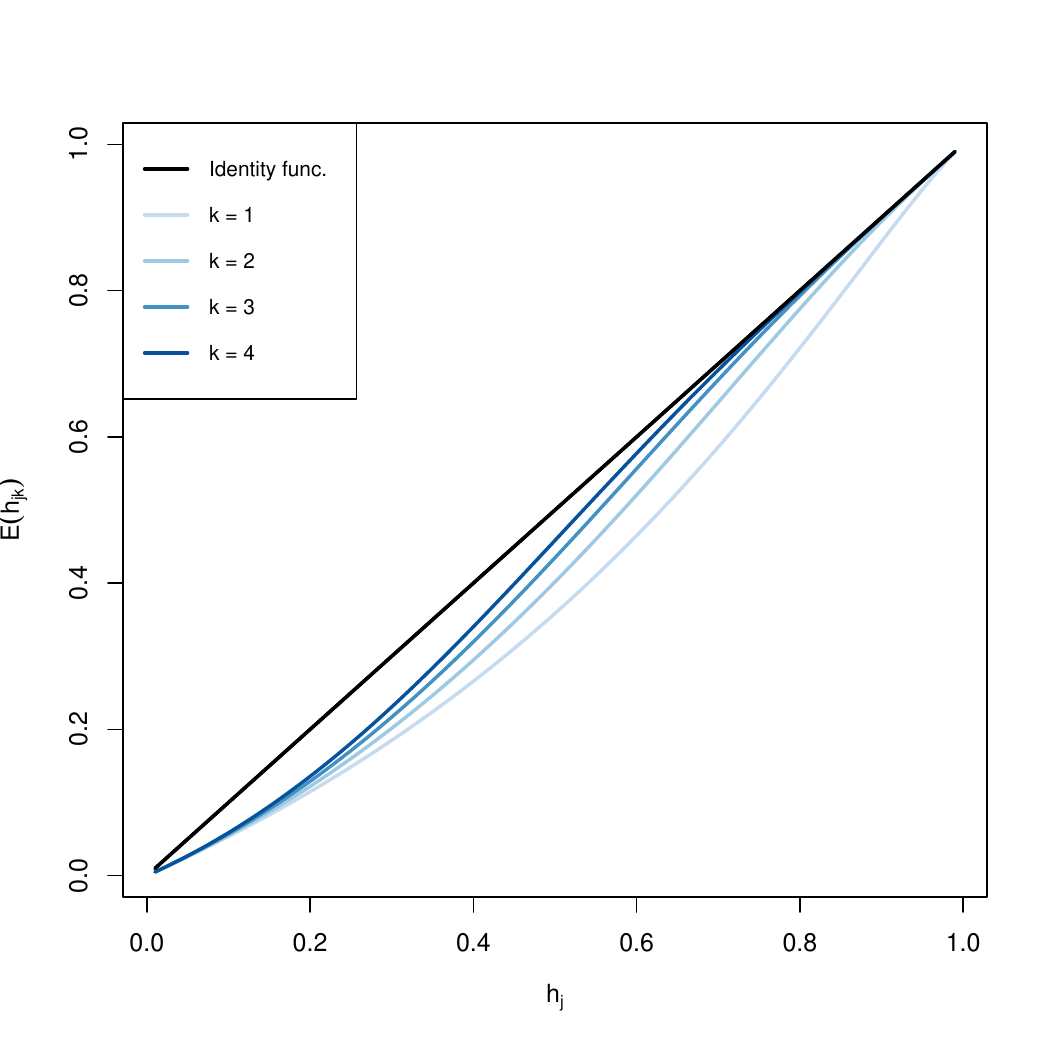}
	\caption{Expectation of $\hat{h}_{j,k}$ as a function of $p=h_j$ for $k=1,2,3,4$ and a Geometric distribution $F$ together with the identity function (black).}
	\label{figure_plot_esperanza}
\end{figure}
\end{example}

\begin{remark}\label{jgrande} The distribution of the estimator $\hat h_{j,k}$ obtained in Theorem \ref{theorem_rhat} is valid as long as $r_n\ge j+k$, that is, the random variables $V_m^\ell$ appearing in the definition of the estimator $\hat h_{j,k}$ in \eqref{eq_estimador} coincide with the corresponding $\tilde V_m^\ell$ in \eqref{tilde_V}. When $r_n <j+k$, the derivation of the properties of $\hat h_{j,k}$ may be carried out taking into account that $V_m^\ell$ is a censored version of $\tilde V_m^\ell$. In practice, this means that the distribution of $\hat h_{j,k}$ follows that of Theorem \ref{theorem_rhat} for all $j\le r_n-k$, but not for $j\in (r_n-k+1,r_n]$. Nevertheless, from our original sample of $\delta$-records with $\delta=-k-1$, we can extract the sample of $\delta'$-records, where $ \delta'=-k'-1$ for $k'<k$. Thus, in the case where $j+k\ge r_n$, to obtain a MLE of $h_j$ with the distribution of Theorem \ref{theorem_rhat}, we may use $\hat h_{j,k'}$ with $k'$ such that $j+k'\le r_n$. In this scenario, all the $V_m^l$ in Equation \eqref{eq_estimador} are equal to the corresponding $\tilde V_m^\ell$ in \eqref{tilde_V}.
\end{remark}

\section{Applications in Statistical Inference}\label{section_practica}

\subsection{Confidence intervals for $h_j$}

We may use the exact distribution of $\hat h_{j,k}$ to construct confidence intervals for $h_j$ by inverting numerically the CDF in Theorem \ref{theorem_rhat}. However, note that the distribution of $\hat h_{j,k}$ depends not only on $h_j$, but also on $d_{j,k} = h_j + \overline{F}(j+k \vert j) = h_j + (1 - h_j)(1 - h_{j+1}) \cdots (1 - h_{j+k}) = h_j + (1 - h_j)\overline{F}(j+k \vert j+1)$. 
Thus, $\overline{F}(j+k \vert j+1)$ may be seen as a nuisance parameter in the construction of a confidence interval. 

Nonetheless, note that if the assumption $h_{j+1}\approx \cdots\approx h_{j+k}$ seems reasonable, then $d_{j,k}\approx h_j+(1-h_j)^{k+1}$. With this approximation, the CDF of $\hat h_{j,k}$ depends only on the parameter $h_j$, so the inversion method may be applied. The assumption is exact for the Geometric distribution, so it should also be valid for distributions where the hazard function does not exhibit abrupt changes. 

\begin{example}[Negative binomial distribution] Let $\{X_n\}_{n\in\N}$ be a sequence of i.i.d. random variables such that $X_1\sim\dbinneg(m;p)$ for some $m\in\N$ and $p\in(0,1)$. 
It is shown in \cite{Vervaat73} that
\begin{equation*}p-\frac{(m-1)(1-p)}j\le h_j\le p.\end{equation*}
This means that the values  $h_{j+1},\ldots,h_{j+k}$ are not very different from $h_j$, so we may use the approximation above. As an example, we fix $m=5$, $p=0.8$ and build confidence intervals for  $h_6$ when $k=3,4$. In Figure \ref{fig_ci_nbinom}, both the estimates $\hat{h}_{6,k}$ and the confidence intervals are plotted for $100$ independent simulated sequences.  We note that the confidence intervals are conservative, which may be due to the approximation of $h_{j+1},\ldots,h_{j+k}$ included in our procedure.
\begin{figure}[H]
\centering
\includegraphics[scale=0.53]{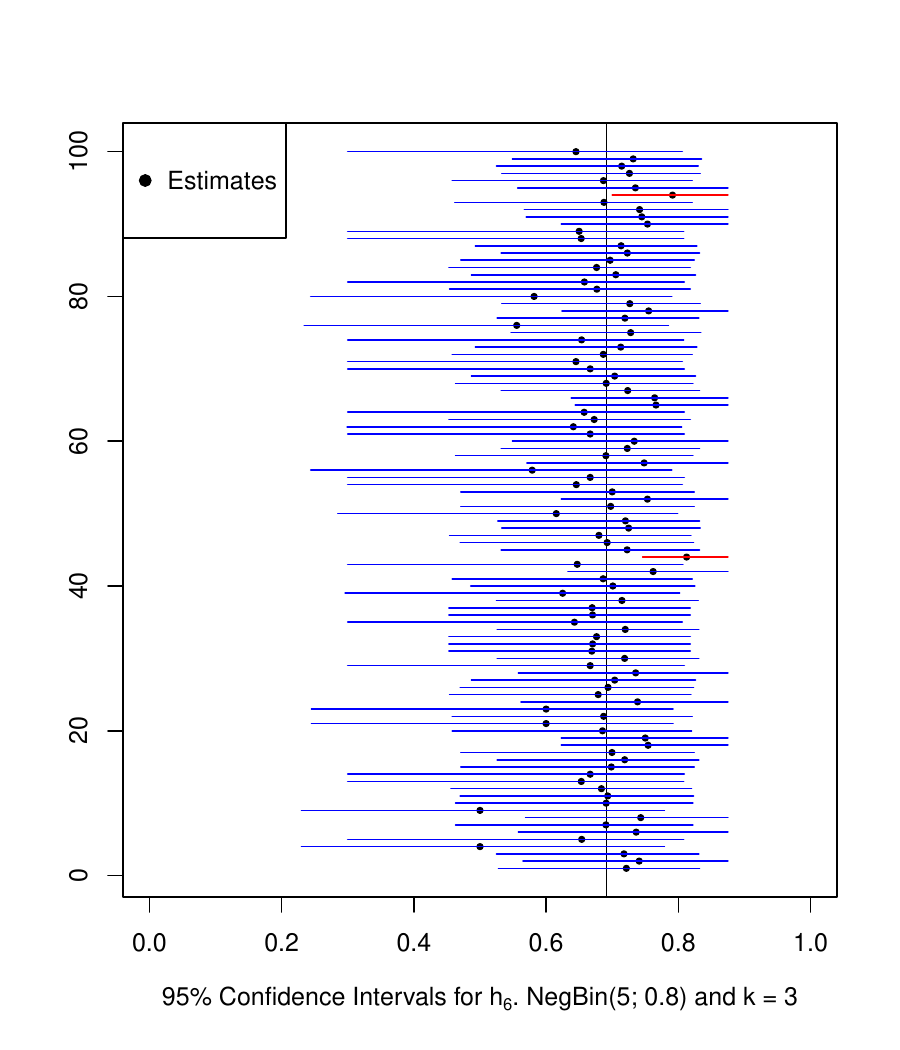}
\includegraphics[scale=0.53]{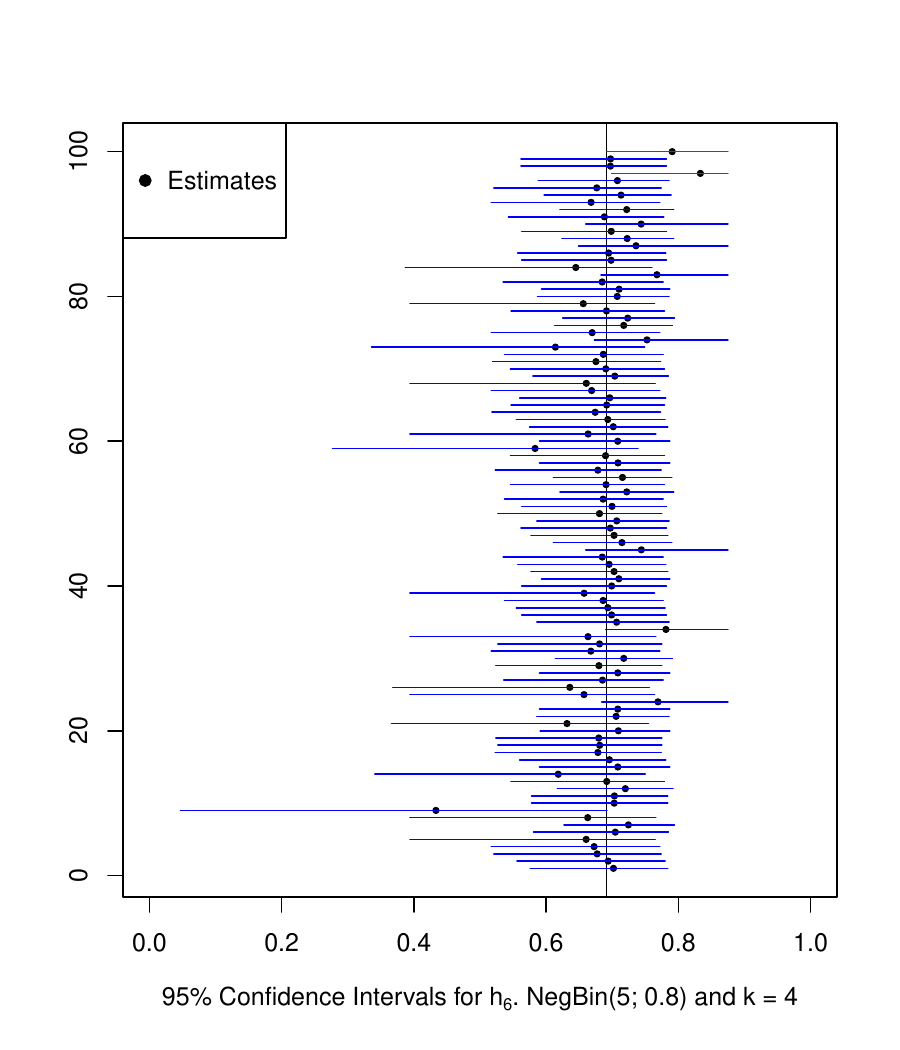}
\caption{Estimates and $95\%$ confidence intervals for $h_6$ for $100$ independent sequences following a negative binomial distribution $\dbinneg(5;0.8)$ and $k=3,4$. The intervals that contain the value $h_6\approx 0.6914$ are plotted in blue and those that do not are plotted in red.}
\label{fig_ci_nbinom}
\end{figure}
\end{example}

\subsection{Likelihood ratio test for goodness of fit}\label{section_test_LR}

Given a sample $\mathbf{t}$ of $\delta$-records, we can use it to test $H_0:F=F_0$ vs $H_1:F\ne F_0$, where $F_0$ is a known distribution. To this end, we may employ the likelihood ratio statistic:
\begin{equation*}
\text{LR}=\frac{\mathcal{L}(\mathbf{t}\mid\mathbf{h}_0)}{\sup_{\mathbf{h}}\mathcal{L}(\mathbf{t}\mid\mathbf{h})},
\end{equation*}
 where the numerator is \eqref{eq_verosimilitud} evaluated at the value of $\mathbf{h}$ corresponding to the hazard rate from the theoretical distribution $F_0$, and the denominator is also \eqref{eq_verosimilitud} but evaluated at the MLE of $\mathbf{h}$ given by \eqref{eq_estimador}.
 
We can also test $H_0:F\in \mathcal{F}_{\boldsymbol{\theta}}$ vs $H_1:F\not\in \mathcal{F}_{\boldsymbol{\theta}}$, where $\mathcal{F}_{\boldsymbol{\theta}}$, $\boldsymbol{\theta}\in\Theta$, is a parametric family of distributions, using 
\begin{equation*}
\text{LR}=\frac{\sup_{\mathbf{h}:F\in \mathcal{F}_{\boldsymbol{\theta}}}\mathcal{L}(\mathbf{t}\mid\mathbf{h})}{\sup_{\mathbf{h}}\mathcal{L}(\mathbf{t}\mid\mathbf{h})}.
\end{equation*}
In this case the numerator is computed as the maximum value of \eqref{eq_verosimilitud} over all $\mathbf{h}$ corresponding to distributions $F\in \mathcal{F}_{\boldsymbol{\theta}}$.

The exact distribution of the likelihood ratio statistic under the null hypothesis is difficult to obtain, even when this hypothesis is simple. We can approximate the p-value of the test by using a standard parametric bootstrap procedure, as outlined in Section 4.4 of \cite{Davison97}. In our setting, if the null hypothesis is simple, we simulate $B$ samples of $\delta$-records from $F_0$. The p-value is then estimated as the proportion of simulations the likelihood ratio statistic obtained from the bootstrap sample is smaller than the one obtained with the original sample. If the null hypothesis is composite, we compute $\hat{\boldsymbol{\theta}}$, the MLE of $\boldsymbol{\theta}\in\Theta$, simulate $B$ samples from $F_{\hat{\boldsymbol{\theta}}}$ and estimate the p-value in the same way.

\begin{example}[Testing the Geometric distribution] As an example of applying the likelihood ratio test based on $\delta$-records, we consider testing $H_0:F\in \mathcal{F}_p$ vs $H_1:F\not\in \mathcal{F}_p$, where $\mathcal{F}_p$ represents the family of Geometric distributions with parameter $p\in(0,1)$. The MLE of $p$ for the Geometric distribution based on $\delta$-records was derived in \cite{Gouet14} and it can be expressed as 
\begin{equation}
\hat{p}_{k} := \frac{\sum\limits_{j=0}^{R_{n}} V_j^k}{\sum\limits_{j=0}^{R_{n}} N_j^k} = \frac{\Delta_k}{\sum\limits_{j=0}^{R_{n}} N_j^k},\label{estimate_2014}
\end{equation}
where $\Delta_k$ is the number of $-(k+1)$-records, that is, $\Delta_k = n + S_1+\cdots+S_n$.
The resulting likelihood ratio is
{\allowdisplaybreaks
\begin{equation*}
\text{LR} = \frac{\left(\sum\limits_{j=0}^{R_{n}} V_j^k\right)^{\sum\limits_{j=0}^{R_{n}} V_j^k} \left(\sum\limits_{j=0}^{R_{n}} (N_j^k - V_j^k)\right)^{\sum\limits_{j=0}^{R_{n}} (N_j^k - V_j^k)}}{\left(\sum\limits_{j=0}^{R_{n}} N_j^k\right)^{\sum\limits_{j=0}^{R_{n}} N_j^k}} \prod_{j=0}^{R_{n}} \frac{(N_j^k)^{N_j^k}}{(V_j^k)^{V_j^k} (N_j^k - V_j^k)^{(N_j^k - V_j^k)}}. 
\end{equation*}
}

For assessing the test, we simulated $1000$ samples of $\delta$-records with $n=3$ from the Geometric and Poisson distributions with random parameters and counted the number of times that our test rejected the null hypothesis in each case. The random parameters are simulated according to the following distributions:
\begin{itemize}
	\item If $F\sim \dgeom(p)$, the probability of success $p$ is drawn from a Uniform distribution on $[0.01, 1)$, that is, $p\sim\mathrm{U}([0.01, 1))$.
	\item If $F\sim \dpoi(\lambda)$, the rate $\lambda$ is drawn from a Gamma distribution with shape and scale parameter equal to $2$, that is, $\lambda\sim\mathrm{Gamma}(2, 2)$.
\end{itemize}

The results are shown in Table \ref{table_power_test} for a significance level of $\alpha=0.05$. We see that for both distributions the type I error is very close to the nominal 0.05 for every value of $k$. This is rather satisfactory, especially considering that the p-value of the tests is not exact but approximated via parametric bootstrap. Regarding type II error, the power of the test is dependent on $k$, with larger values of $k$ resulting in more powerful tests due to the inclusion of more observations.

\begin{table}[H]
	\centering
	\begin{tabular}{|c|c|c|c|c|}
		\cline{2-5}
		\multicolumn{1}{c|}{}&\multicolumn{2}{c|}{$H_0: F$ is Geometric} & \multicolumn{2}{c|}{$H_0: F$ is Poisson}\\
		\hline
		\multirow{4}{*}{Sample is Geometric} & $k$ & Rejection rate (Type I Error) & $k$ & Rejection rate (Type II Error)\\
		\cline{2-5}
		&1 & 0.065 & 1 & 0.268\\
		&2 & 0.064 & 2 & 0.393\\
		&3 & 0.053 & 3 & 0.594\\
		\hline
		\multirow{4}{*}{Sample is Poisson}& $k$ & Rejection rate (Type II Error) & $k$ & Rejection rate (Type I Error)\\
		\cline{2-5}
		&1 & 0.302 & 1 & 0.056\\
		&2 & 0.512 & 2 & 0.051\\
		&3 & 0.599 & 3 & 0.045\\
		\hline
	\end{tabular}
	\caption{Rejection rates of the $\delta$-record-based goodness-of-fit test computed for $1000$ randomly-generated Geometric and Poisson sequences with $n = 3$.}
	\label{table_power_test}
\end{table}

\end{example}
\subsection{Hazard rate estimator under monotonicity assumptions}

So far we have not made any assumption on the form of the hazard rate function $\mathbf{h}$. However, in many situations we may have prior information about the distribution and, in particular, we may assume that it has an increasing failure rate (IFR) or a decreasing failure rate (DFR).

In this section, we derive estimators for $h_j$ under both increasing and decreasing monotonicity constraints. These estimators are obtained by optimizing the likelihood function \eqref{eq_verosimilitud} subject to the conditions $0 \leq h_0 \leq h_1 \leq \cdots < 1$ for increasing hazard rates, or $1 > h_0 \geq h_1 \geq \cdots \geq 0$ for decreasing hazard rates. The problem of maximizing a likelihood function of the form
$$
\prod_{j=1}^n p_j^{\alpha_j}(1-p_j)^{\beta_j}
$$
under a monotonicity assumption is solved in \cite{Ayer55}. Using their result and the product form of the likelihood in \eqref{eq_verosimilitud}, we can readily derive that the increasing and decreasing monotonic hazard rate estimators, $\hat{h}_{j,k}^{\text{inc.}}$ and $\hat{h}_{j,k}^{\text{dec.}}$, are given by
\begin{equation*}
\hat{h}_{j,k}^{\text{inc.}} = \max_{0\le \ell\le j}\ \min_{j\le m\le r_{n}}\frac{\sum\limits_{i = \ell}^m V_{i}^k}{\sum\limits_{i = \ell}^m N_{i}^k},\qquad\qquad
\hat{h}_{j,k}^{\text{dec.}} = \min_{0\le \ell\le j}\ \max_{j\le m\le r_{n}}\frac{\sum\limits_{i = \ell}^m V_{i}^k}{\sum\limits_{i = \ell}^m N_{i}^k}.
\end{equation*}

Figure \ref{fig-poisson-mon} illustrates the increasing trend estimator for a simulated sample from a Poi(6) distribution  with $n=5$ and $k=3$. We observe a very good agreement between the estimated values of $h_j$ and the true values. 

\begin{figure}[h]
	\centering
	\includegraphics[scale=0.55]{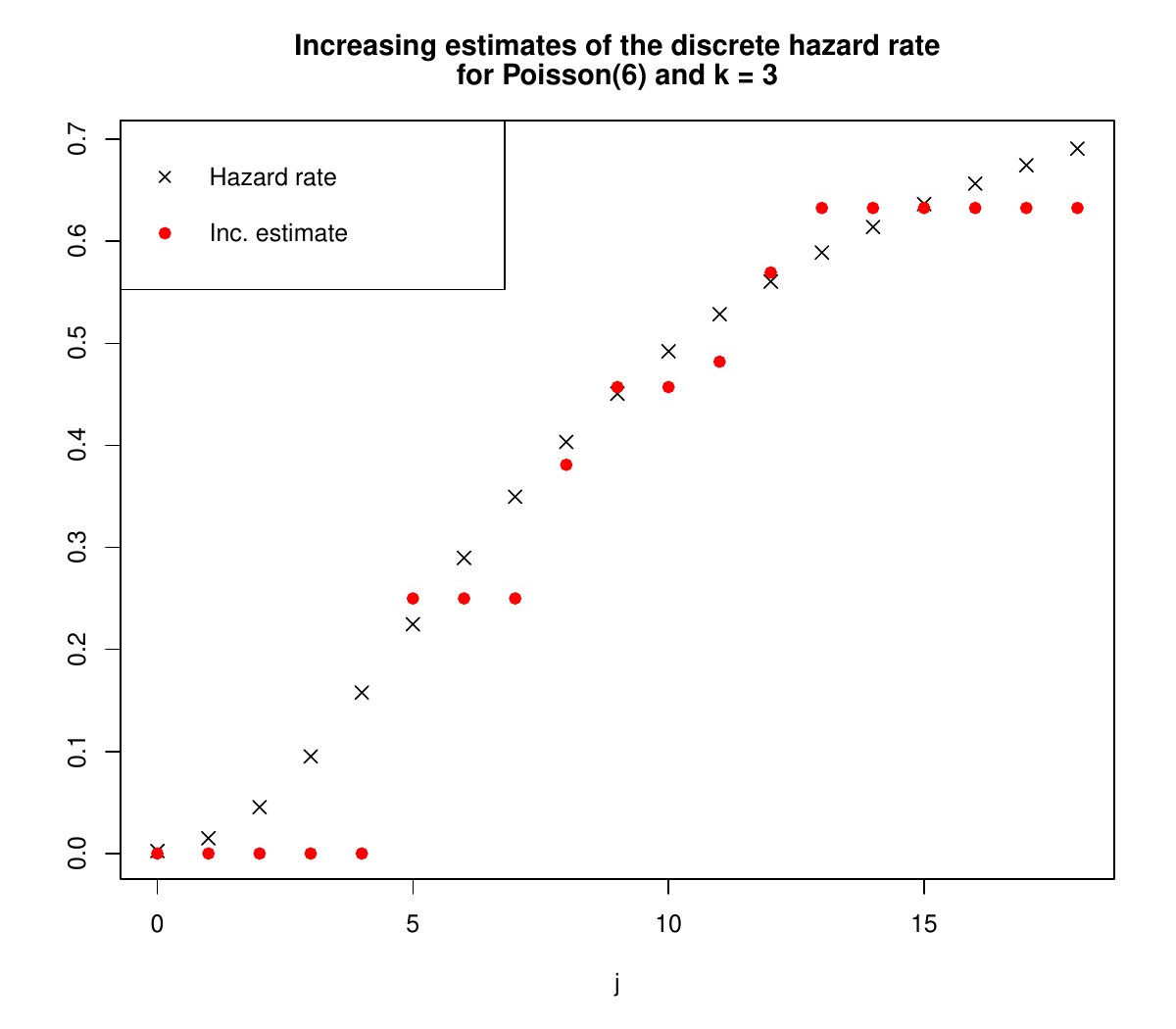}
	\caption{Estimator $\hat{h}_{j,k}^{\text{inc.}}$ for $X_1\sim\dpoi(6)$ with $n = 5$, $k = 3$ and $r_{n} = 18$.}
	\label{fig-poisson-mon}
\end{figure}

\section{Application to real data} \label{section_fixed_sample_size}

In this section, we apply our results to two real data sets. Note that, when dealing with real data, we cannot be certain if the sample includes all the near-records associated with the last record. Therefore, in the likelihood \eqref{eq_vero_2014} the term $\prob(S_n=s_n)$ should be replaced by $\prob(S_n\ge s_n)$ (see \cite{Gouet14}), which is finally tantamount to replacing the rightmost term $\overline{F}(r_n)$ in \eqref{eq_vero_2014} with $\overline{F}(r_n-k-1)$, yielding

\begin{align}
\mathcal{L}(\mathbf{t}) &= \overline{F}(r_n - k - 1)\prod_{i=1}^n \frac{P(r_i)}{\overline{F}(r_i-k-1)^{s_i+1}} \prod_{j=1}^{s_i} P(y^j_i)\nonumber \\
&= \prod_{j=1}^{r_{n}-k-1} h_j^{v^k_j} (1-h_j)^{1 + \sum\limits_{\ell=1}^k v_{j+\ell}^{k-\ell}} \prod_{j=r_{n}-k}^{r_{n}} h_j^{v^k_j} (1-h_j)^{\ \sum\limits_{\ell=1}^k v_{j+\ell}^{k-\ell}}. \label{eq_incomplete_likelihood}
\end{align}

From \eqref{eq_incomplete_likelihood}, we deduce that the estimator $\hat{h}_{j,k}$ for values $j=0, \dots, r_{n}-k-1$ matches the expression in \eqref{eq_estimador} and follows the distribution given in Theorem \ref{theorem_rhat} . This is expected, as $r_{n}$ has been observed, ensuring that all necessary information for these values has been accounted for. For the remaining values of $j$, the estimator becomes
\begin{equation*}
\hat{h}_{j,k} = \frac{V_j^k}{V_j^k + \cdots + V_{j+k}^0}.
\end{equation*}

The modification of the likelihood in this situation results in the MLE for $h_{r_n}$ being trivially equal to 1. This is because $V_{r_n+1}^{k-1}= \cdots= V_{r_n+k}^0=0$. This outcome is intuitive: since from our sample it is uncertain whether a value greater than $r_n$ will occur in the future, the most plausible conclusion is that the support of $F$ ends at $r_{n}$.

\subsection{Real data: times until defective items}
We apply our results to the data presented in Table \ref{table_xiegoh}, taken from Table III in \cite{Xie93}. It corresponds to the number of items observed until a defective item is detected and is modeled by the Geometric distribution with a starting point of $1$. The sample consists of $87$ observations, and $5$ records were identified. The values of the $\delta$-records for $k=1,2,3$ are displayed in Table \ref{table_delta_records_xiegoh}. The resulting estimators $\hat{h}_{j,k}$ are presented in Table \ref{table_estimators_xiegoh}.

\begin{table}[H]
\centering
\begin{tabular}{l}
\hline 1, 5, 1, 1, 1, 3, 2, 4, 3, 2, 3, 1, 1, 1, 3, 1, 3, 1, 6, 4, 1, 9, 2, 6, 2, 1, 3, 1, 3, 1, 1, 10,
2, 7, 1, 8, 1, 1, 2, 1, 1, 6, 1, 2, \\ 1, 4, 1, 1, 1, 3, 5, 1, 1, 1, 1, 5, 2, 4, 5, 1, 2, 2, 1, 3, 1, 
1, 1, 3, 1, 2, 1, 1, 1, 1, 1, 1, 5, 2, 2, 4, 6, 1, 3, 1, 1, 1, 1\\
\hline
\end{tabular}
\caption{Real observed data. Xie and Goh (1993).}
\label{table_xiegoh}
\end{table}

\begin{table}[H]
	\centering
	\begin{tabular}{rlc}
		\hline $k$ & $-(k+1)$-records& Total \\
		\hline 1 & $1,5,4,6,9,10$ & 6\\
		2 & $1,5,3,4,3,3,3,3,6,4,9,10,8$ & 13\\
		3 & $1,5,3,2,4,3,2,3,3,3,6,4,9,6,10,7,8$ & 17\\
		\hline
	\end{tabular}
	\caption{$\delta$-records for the sample in Table \ref{table_xiegoh}.}
	\label{table_delta_records_xiegoh}
\end{table}

\begin{table}[H]
\centering
\begin{tabular}{|c|c|c|c|c|c|c|c|c|c|c|c|}
\hline
\backslashbox{$k\!\!\!$}{$\!\!\!j$} & 1 & 2 & 3 & 4 & 5 & 6 & 7 & 8 & 9 & 10 & N\\
\hline 1 & 0.500 & 0.000 & 0.000 & 0.333 & 0.333 & 0.500 & 0.000 & 0.000 & 0.333 & 1 & 6\\
2 & 0.500 & 0.000 & 0.625 & 0.400 & 0.333 & 0.500 & 0.000 & 0.250 & 0.500 & 1 & 13\\
3 & 0.500 & 0.200 & 0.500 & 0.400 & 0.333 & 0.500 & 0.200 & 0.333 & 0.500 & 1 & 17\\
\hline
\end{tabular}
\caption{Estimated values $\hat{h}_{j,k}$ for $j=1,\ldots,10$ and $k=1,2,3$ for the $\delta$-record samples in Table \ref{table_delta_records_xiegoh}.}
\label{table_estimators_xiegoh}
\end{table}

Additionally, we can apply our goodness-of-fit test in Section \ref{section_test_LR} to assess whether Xie and Goh's assumption that the sample follows a Geometric distribution is reasonable. The test is applied for $k = 1, 2, 3$ (see Table \ref{table_p_values}). We conclude that the hypothesis cannot be rejected. 

\begin{table}[H]
\centering
\begin{tabular}{|c|c|}
\hline
 \multicolumn{2}{|c|}{$H_0: F$ is Geometric}\\
\hline
$k$ & p-value \\
\hline 1 & 0.961\\
2 & 0.615  \\
3 & 0.918 \\
\hline
\end{tabular}
\caption{Simulated p-values of the goodness-of-fit test for the sample in Table \ref{table_xiegoh}.}
\label{table_p_values}
\end{table}
 In \cite{Gouet14}, the estimator in \eqref{estimate_2014} is applied to this data set and compared with the MLE for the entire sample, which is simply the inverse of the mean of the sample, yielding a value of $0.42$. Similarly, in Figure \ref{fig-estiamtors-xiegoh} of this paper, the values of $\hat{h}_{j,k}$ presented in Table \ref{table_estimators_xiegoh} are plotted alongside this estimate. It is noteworthy that, for $k=3$, the estimates $\hat{h}_{j,k}$ are close to $0.42$, the parametric MLE of the constant hazard rate, despite the fact that our estimation procedure is nonparametric and utilizes only around $20\%$ of the whole sequence of observations.

\begin{figure}[h]
	\centering
	\includegraphics[scale=0.58]{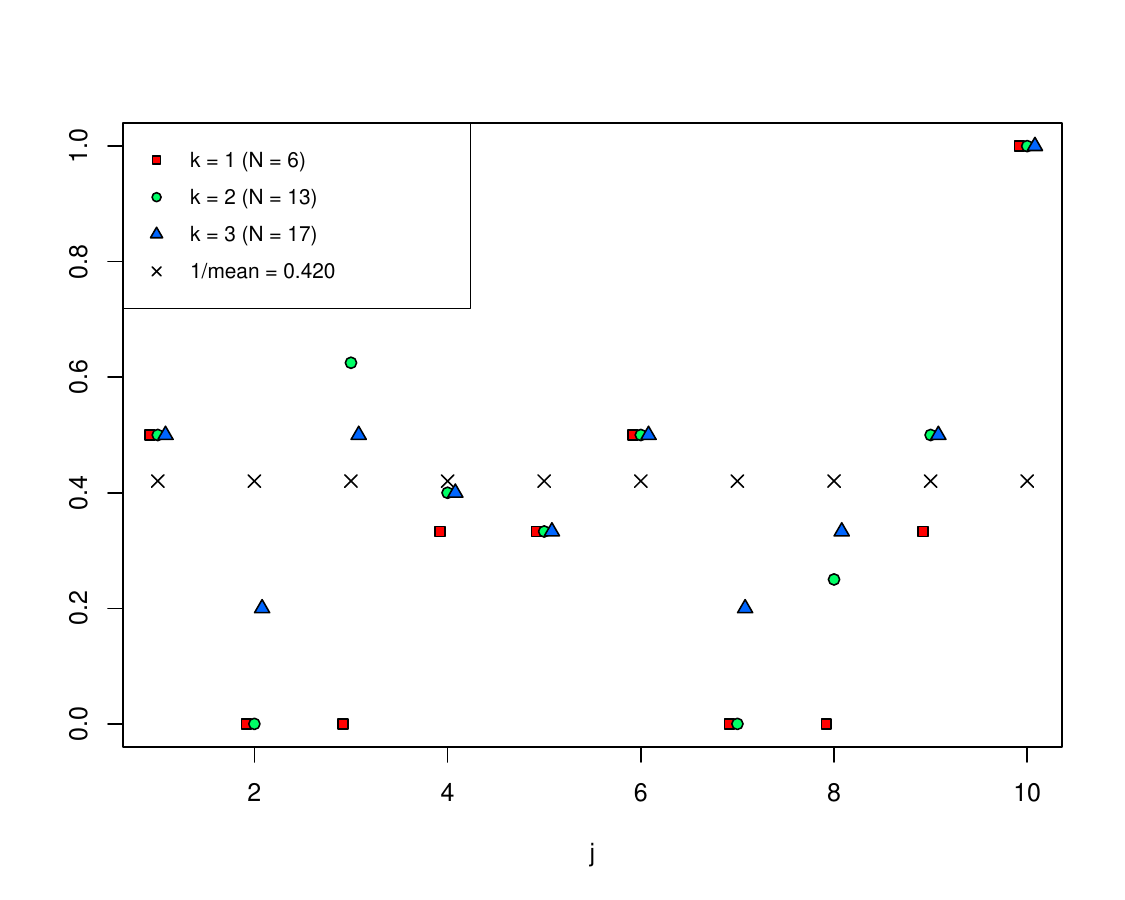}
	\caption{Values of the estimators $\hat{h}_{j,k}$ in Table \ref{table_estimators_xiegoh} along with the MLE of the probability of success $p$ for the entire sample.}
	\label{fig-estiamtors-xiegoh}
\end{figure}

\subsection{Real data: occurrence of high-magnitude earthquakes}
In the study of the temporal occurrence of earthquakes, the Poisson distribution has been widely used as a modeling tool. In \cite{Wu19}, the authors conducted a chi-square goodness-of-fit test for the number of earthquakes with magnitudes greater than $7.5$ per year, finding good agreement between the expected frequencies under the Poisson model and the observed data from the worldwide earthquake catalog of the United States Geological Survey (USGS) for the period 1980--2015.

In our analysis, we use an extended data set covering the period from 1950 to 2023. The corresponding data are presented in Table \ref{data_earthquakes}, and a comparison between the expected and observed frequencies is plotted in Figure \ref{poisson_prob_earthquakes}, demonstrating that the Poisson distribution remains a reasonable model for this broader time range.

\begin{table}[H]
	\centering
	\begin{tabular}{l}
		\hline 7, 3, 3, 2, 1, 1, 2, 6, 4, 1, 6, 2, 0, 8, 3, 6, 5, 0, 6, 4, 
		2, 5, 4, 3, 2, 8, 4, 6, 4, 3, 2, 2, 0, 1, 2, 3, 3, 2, 3, 2, \\
		5, 3, 2, 4, 5, 5, 6, 4, 4, 4, 7, 7, 6,  5, 3, 5, 5,  10,  2, 8, 
		6, 4, 6, 6, 5, 8, 7, 3, 7, 3, 4, 6, 2, 6
		\\
		\hline
	\end{tabular}
	\caption{Temporal occurrence of earthquakes with magnitudes of 7.5 and above worldwide from 1950 to 2023. United States Geological Survey (USGS).}
	\label{data_earthquakes}
\end{table}
\begin{figure} [H]
	\centering
	\includegraphics[scale=0.65]{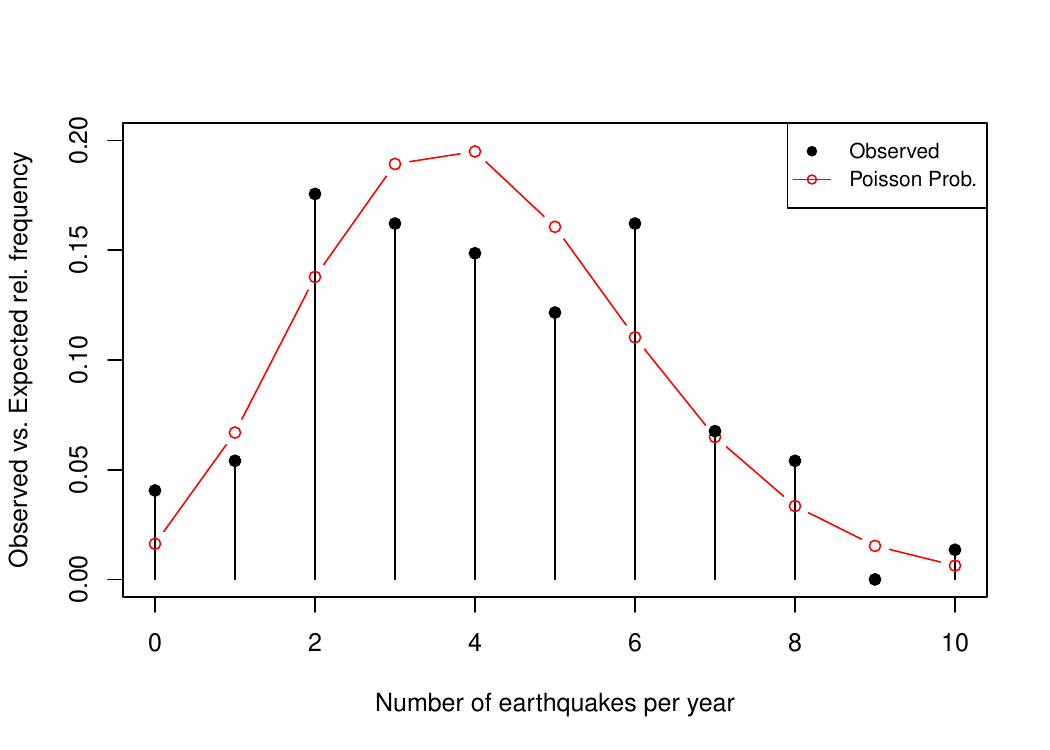}
	\caption{Comparison between observed relative frequency of earthquakes with magnitudes of 7.5 and above from 1950 to 2023 (Table \ref{data_earthquakes}) and the adjusted Poisson model.}
	\label{poisson_prob_earthquakes}
\end{figure}

In addition, we also apply our goodness-of-fit test to assess whether this assumption is compatible with the collected data. The resulting simulated p-values are presented in Table \ref{table_earthquakes_p_values}. As before, we conclude that we cannot reject the hypothesis that the number earthquakes with magnitudes above $7.5$ per year follows a Poisson distribution.

\begin{table}[H]
	\centering
	\begin{tabular}{|c|c|}
		\hline
		 \multicolumn{2}{|c|}{$H_0: F$ is Poisson}\\
		\hline
		$k$ & p-value \\
		\hline 1 & 0.906\\
		2 & 0.361\\
		3 & 0.812\\
		4 & 0.888\\
		\hline
	\end{tabular}
	\caption{Simulated p-values of the goodness-of-fit test for the sample in Table \ref{data_earthquakes}.}
	\label{table_earthquakes_p_values}
\end{table}

 Given that the Poisson distribution has increasing hazard rates, we applied our increasing hazard rate estimator to this data set. The observed $\delta$-records and the resulting estimates are shown in Tables \ref{delta-records-earthquakes} and \ref{table_estimators_earthquakes} respectively. Note that, since $x_1=7$, our sample of $\delta$-records cannot have values smaller than $7-k$, so the estimates of $h_j$ with $j < 7 - k$ are trivially 0.

\begin{table}[H]
	\centering
	\begin{tabular}{rlc}
		\hline $k$ & $-(k+1)$-records & Total\\
		\hline 1 & $7,6,6, 8, 8, 7,7,10$ & 8\\
		2 & $7,6,6,8,6, 6,8,6,6,7,7,6,10,8,8$ & 15\\
		3 & $7,6,4,6,8,6,5,6,5,8,6,5,5,5,6,7,7,6,5,5,5,10,8,8,7,7$ & 26\\
		4 & $7, 3, 3, 6, 4, 6, 8, 6, 5, 6, 4, 5, 4, 8, 4, 6, 4, 5, 4, 5, 5, 6, 4, 4, 4, 7, 7,$& 41\\ 
		&$6, 5, 5, 5, 10, 8, 6, 6, 6, 8, 7, 7, 6, 6$ & \\
		\hline
	\end{tabular}
	\caption{$\delta$-records for the sample in Table \ref{data_earthquakes}.}
	\label{delta-records-earthquakes}
\end{table}

\begin{table}[H]
\centering
\begin{tabular}{|c|c|c|c|c|c|c|c|c|c|c|c|c|}
\cline{2-13}
\multicolumn{1}{c|}{}&\multicolumn{12}{c|}{Nonparametric increasing estimates of the discrete hazard rate}\\
\hline
\backslashbox{ $k\!\!\!$}{$\!\!\!j$} & 0 & 1 & 2 & 3 & 4 & 5 & 6 & 7 & 8 & 9 & 10 & N\\
\hline 
1 & 0.000 & 0.000 & 0.000 & 0.000 & 0.000 & 0.000 & 0.500 & 0.500 & 0.600 & 0.600 & 1 & 8\\
 2 & 0.000 & 0.000 & 0.000 & 0.000 & 0.000 & 0.000 & 0.538 & 0.538 & 0.714 & 0.714 & 1 & 15\\
 3 & 0.000 & 0.000 & 0.000 & 0.000 & 0.200 & 0.381 & 0.538 & 0.538 & 0.714 & 0.714 & 1 & 26\\
4 & 0.000 & 0.000 & 0.000 & 0.286 & 0.300 & 0.381 & 0.545 & 0.545 & 0.714 & 0.714 & 1 & 41\\
\hline
\addlinespace
\cline{2-13}
\multicolumn{1}{c|}{}& \multicolumn{12}{c|}{Parametric (Poisson) estimates of the discrete hazard rate}\\
\hline &0.016 & 0.069 & 0.151 & 0.244 & 0.332 & 0.409 & 0.473 & 0.527 & 0.572 & 0.610 & 0.642 & 74\\
\hline
\end{tabular}
\caption{Estimated values $\hat{h}^{\text{inc.}}_{j,k}$ for $j=0,\ldots,10$ and $k=1,2,3,4$ for the $\delta$-record samples in Table \ref{delta-records-earthquakes} and nonparametric increasing estimators of $h_j$ for the entire sample $(k=\infty)$ in Table \ref{data_earthquakes}.}
\label{table_estimators_earthquakes}
\end{table}

The last line of Table \ref{table_estimators_earthquakes} shows the parametric estimation of $h_j$, assuming that the sample comes from a Poisson distribution and using all the data. That is, the MLE of $\lambda$ is $\overline x=4.108$ and the estimates of the hazard function $h_j$ are the corresponding values of $h_j$ from a Poi(4.108) distribution. See also Figure \ref{fig_earthquakes_estimators} for the comparison of the estimates. It is remarkable that, for most of the values of $k$, the nonparametric estimates are close to the parametric ones, except for the first and last values of $j$, for which we have the least information.

\begin{figure}[H]
\centering
\includegraphics[scale=0.58]{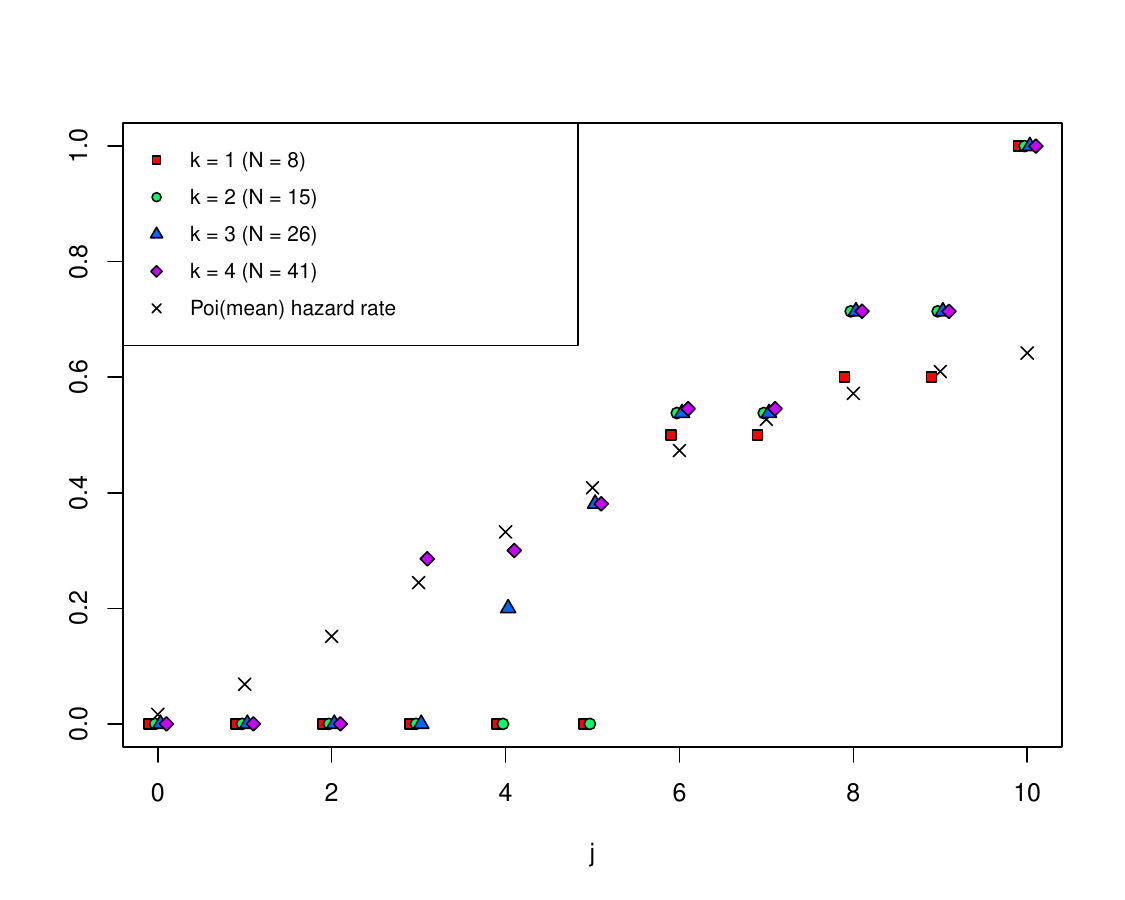}
\caption{Value of the estimators $\hat{h}_{j,k}$ and parametric estimates of the Poisson hazard rates in Table \ref{table_estimators_earthquakes}.}
\label{fig_earthquakes_estimators}
\end{figure}

\section{Conclusions and future work}\label{conc}
Statistical inference based on records is a well-established topic in the literature. The inclusion of near-records in the sample has been shown to be advantageous in parametric inference (\cite{Gouet14}, \cite{Gouet20}, \cite{LBSM13}). This is the first paper to address nonparametric inference based on $\delta$-records. Our results on the estimation of the discrete hazard rate, which include the explicit form of the MLE and its exact distribution, along with confidence intervals and hypothesis tests, demonstrate that near-records can be successfully incorporated into the estimation procedure in this context. This is particularly important in practice as $\delta$-records can be collected by a slight modification of the experimental setup for records (see \cite{Gouet14}), eliminating the need for collecting several samples of records, as required in traditional nonparametric procedures based solely on record values (see \cite{Gulati94}).

There are some directions for future work on this topic, such as:
\begin{itemize}
	\item Continuous distributions. This paper focuses on discrete distributions and the MLE is derived for a finite number of parameters $(h_1,\ldots,h_{r_n})$. To extend this approach to continuous distributions, the data can be discretized and our results can be used to estimate the hazard rate of the discrete version of the distribution. Then, estimators of the continuous hazard rate can be obtained via smoothing techniques.
	
	\item Bayesian inference. In this paper, we have only considered maximum likelihood estimation, which is a frequentist procedure. Since we have the explicit expression for the likelihood, we may include prior distributions for $\mathbf{h}$ and develop Bayesian procedures. This approach has proven successful in parametric estimation based on $\delta$-records both for discrete and continuous random variables (\cite{Gouet14}, \cite{Gouet20}).
	
	\item Incorporating record- and near-record-times in the sample. While including these statistics in the likelihood \eqref{eq_vero_2014} is straightforward, the maximization procedure to find the MLEs is much more challenging and, in particular, no explicit expression for the estimators seems to exist. This complicates the study of its properties and the construction of confidence intervals, for instance. However, since the information provided by record- and near-record-times (if they are collected) can be valuable, the additional complexity in the estimation process is justified.
\end{itemize}
	
\section*{Ackowledgements}

This research was partially funded by grants PID2020-116873GB-I00 and TED2021-130702B-I00 from the Ministry of Science, Innovation and Universities of Spain. The authors are members of the research group Modelos Estoc\'{a}sticos (E46\_23R), Gobierno de Arag\'{o}n. M. Alcalde gratefully acknowledges the support by the doctoral scholarship ORDEN ECU/592/2024 from Gobierno de Arag\'{o}n.

\section{Appendix}

\begin{lemma}[Lemma 1 in \cite{Stepanov93}] \label{lemma_Stepanov} 
Let $\{X_n\}_{n\in\N}$ be an i.i.d. sequence with values in $\Z_+$ and infinite support. Then the random variables $\{A_j^0\}_{j\in\Z_+}$ are independent and
$$
\prob(A^0_j = m) = h_j^m(1-h_j),\quad m \in\Z_+,
$$
i.e.,
$A^0_j\sim\mathrm{Geom}^*(1- h_j)$ for all $j\in\Z_+$.
\end{lemma}

\begin{definition}[\cite{Johnson97}]\label{def_geometric}
A discrete random vector $\mathbf{Z} = (Z_1,\ldots, Z_n)$ has a $n$-multidimensional Geometric distribution (with support $\Z_+^n$) with parameters $\pi_1,\ldots,\pi_n\in[0,1)$ such that $\sum_{i=1}^n \pi_i <1$ if its probability mass function is 
$$
\prob(Z_1 = m_1,\ldots,Z_n = m_n ) = \pi_1^{m_1}\cdots \pi_n^{m_n}\rho\ {m_1+\cdots+m_n\choose m_1,\ldots,m_n},
$$
where $\rho = 1-\sum_{i=1}^n \pi_i$, for all $m_1,\ldots,m_n\in\Z_+$. This is denoted as $\mathbf{Z}\sim\mathrm{Geom}^*(\pi_1,\ldots,\pi_n)$.
\end{definition}

\begin{lemma}\label{lemma_combinatoria}
Let $n_1, n_2, \dots, n_{k+1} \in \mathbb{Z}_+$ with $n_{k+1} \geq 1$. The following identity holds:
\begin{align*}
\sum_{i_1=0}^{n_1}\sum_{i_2=0}^{n_2} \cdots \sum_{i_k=0}^{n_k} &\binom{i_1+i_2+\cdots+i_k}{i_1, i_2, \dots, i_k} \binom{n_1-i_1+n_2-i_2+\cdots+n_k-i_k+n_{k+1}-1}{n_1-i_1, n_2-i_2, \dots, n_k-i_k, n_{k+1}-1} \\
&= \binom{n_1+n_2+\cdots+n_k+n_{k+1}}{n_1, n_2, \dots, n_k, n_{k+1}}.
\end{align*}
\begin{proof}
	Consider a collection of sets $A_1,\ldots,A_{k+1}$, each containing indistinguishable individuals. First, choose an element $a$ from set $A_{k+1}$. Next, select subsets $I_1,\ldots,I_k$ from $A_1,\ldots,A_{k}$, respectively, which represent the individuals to be placed before $a$. Let $i_\ell$ denote the number of individuals in subset $I_\ell$ for $\ell=1,\ldots,k$. The multinomial coefficient
	$$
	\binom{i_1+i_2+\cdots +i_k}{i_1,i_2,\ldots,i_k}
	$$
	counts the number of ways to arrange the individuals in subsets $I_1,\ldots,I_k$. Once these individuals are placed before $a$, the remaining individuals can be arranged in
	$$
	 \binom{n_1-i_1+n_2-i_2+\cdots+n_k-i_k+n_{k+1}-1}{n_1-i_1, n_2-i_2, \dots, n_k-i_k, n_{k+1}-1}
	$$
	ways. Finally, multiplying these multinomial coefficients and summing over all possible choices of subsets $I_1,\ldots,I_k$ accounts for all possible arrangements of the individuals from sets $A_1,\ldots,A_{k+1}$, thereby yielding the desired result.
\end{proof}
\end{lemma}

\begin{lemma}\label{propdistrgeom} Let $\mathbf{Z} = (Z_1, \dots, Z_n) \sim \mathrm{Geom}^*(\pi_1, \dots, \pi_n)$, then
	$$
	\left(Z_1,\sum_{i=2}^n Z_i\right)\sim\dgeom\left(\pi_1,\sum_{i=2}^n \pi_i\right).
	$$
\begin{proof}
	It is straightforward from the interpretation of the multidimensional Geometric distribution. Each component $Z_i$ corresponds to the number of failures of type $i$ (which occurs with probability $\pi_i$) before the first success in independent trials. Consequently, summing the last $n-1$ components corresponds to aggregating their respective types of failure. The probability of this new kind of failure is $\pi_2+\cdots+\pi_n$.
\end{proof}
\end{lemma}

\begin{lemma}\label{lemma_diofanto}
Let \( a, b \in \mathbb{Z}_+ \) with \( a < b \) and \( \gcd(a, b) = 1 \). The nonnegative integer solutions to the linear Diophantine equation
\[
ax - by = -a
\]
are given by
\[
(x, y) = (-1 + bi, ai), \quad i \in \mathbb{N}.
\]
\begin{proof}
The result is obvious.
\end{proof}
\end{lemma}

\begin{lemma} \label{lemma_suma_series_potencias} 
Let $m\in\mathbb{Z}_+$ and $x\in(0,1)$.
\begin{enumerate} [a)]
\item \label{lemma_ssp_aparatado_a} 
\[
\sum_{\ell=0}^{\infty} {m+\ell\choose \ell}\ x^\ell\frac{1}{\ell+m+1} = \frac{1}{x^{m+1}}\int_{0}^{x}\frac{u^m}{(1-u)^{m+1}}\mathrm{d}u.
\]
\item \label{lemma_ssp_aparatado_b} 
\begin{align*}
\sum_{\ell=0}^\infty {m+\ell\choose \ell}\ x^\ell \frac{\ell^2}{(\ell+m+1)^2} &= \frac{(m+1)}{x^{m+1}}\int_0^x \frac{u^{m+1}(\log{x} - \log{u})}{(1-u)^{m+3}} \mathrm{d}u \\ 
&\qquad + \frac{(m+1)^2}{x^{m+1}}\int_0^x \frac{u^{m+2}(\log{x} - \log{u})}{(1-u)^{m+3}} \mathrm{d}u.
\end{align*}
\end{enumerate}

\begin{proof}
\begin{enumerate}[a)]
\item Consider the function \( f(x) = \sum_{\ell \geq 0} \binom{m+\ell}{\ell}\ x^\ell \frac{1}{\ell+m+1} \) and define \( g(x) = x^{m+1} f(x) \). First, observe that
\[
g'(x) = x^m \sum_{\ell=0}^\infty \binom{m+\ell}{\ell}\ x^\ell = \frac{x^m}{(1-x)^{m+1}} \mathbb{P}(Y \geq 0),
\]
where \( Y \sim \mathrm{NegBin}(m+1; 1-x) \). This gives us
\[
g'(x) = \frac{x^m}{(1-x)^{m+1}}.
\]
Now, since \( g(0) = 0 \), integrating both sides yields
\[
x^{m+1} f(x) = g(x) = \int_0^x \frac{u^m}{(1-u)^{m+1}} \, \mathrm{d}u.
\]

\item A similar approach can be applied by defining \( f(x) = \sum_{\ell \geq 0} \binom{m+\ell}{\ell}\ x^\ell \frac{\ell^2}{(\ell+m+1)^2} \) and \( g(x) = x^{m+1} f(x) \). Differentiating \( g(x) \), we obtain
\[
g'(x) = x^{-1} \sum_{\ell=0}^\infty \binom{m+\ell}{\ell}\ x^{\ell+m+1} \frac{\ell^2}{\ell+m+1}.
\]
Next, define \( h(x) = x g'(x) \), and upon differentiating, we have
\[
h'(x) = \frac{x^m}{(1-x)^{m+1}} \mathbb{E}(Y^2),
\]
where \( Y \sim \mathrm{NegBin}(m+1; 1-x) \). Thus,
\[
h'(x) = \frac{(m+1) x^{m+1}}{(1-x)^{m+3}} (1 + (m+1)x).
\]
Given that \( h(0) = 0 \), it follows that
\[
x g'(x) = h(x) = \int_0^x \frac{(m+1) u^{m+1}}{(1-u)^{m+3}} (1 + (m+1)u) \, \mathrm{d}u.
\]
Finally, since \( g(0) = 0 \), we have
\[
x^{m+1} f(x) = g(x) = \int_0^x \frac{1}{v} \int_0^v \frac{(m+1) u^{m+1}}{(1-u)^{m+3}} (1 + (m+1)u) \, \mathrm{d}u \, \mathrm{d}v.
\]
Applying Fubini’s Theorem to the right-hand side and solving for \( f(x) \) yields the desired result.

\end{enumerate}
\end{proof}
\end{lemma}
\end{document}